\documentclass{article}
\usepackage[margin=1in] {geometry}
\usepackage[utf8]{inputenc}

\usepackage{float}
\usepackage{makeidx}
\usepackage{layout}
\usepackage{array}
\usepackage{amssymb}
\usepackage{amsmath, longtable}
\usepackage{mathtools}
\usepackage{tabularx} 
\usepackage{caption}
\usepackage[utf8]{inputenc}
\usepackage{fancyhdr}
\usepackage{amsthm}
\usepackage{comment}
\usepackage{graphicx}
\usepackage{scalerel}
\usepackage[linesnumbered,boxed,ruled,vlined,algo2e]{algorithm2e}
\usepackage{subcaption}
\usepackage{bbm}
\usepackage{arydshln}
\usepackage{enumitem}
\usepackage{algorithm}
\usepackage{algorithmic}

\usepackage[numbers,comma,sort&compress]{natbib}
\makeatletter
\def\namedlabel#1#2{\begingroup
    #2%
    \def\@currentlabel{#2}%
    \phantomsection\label{#1}\endgroup
}
\makeatother

\newtheorem{definition}{Definition}
\newtheorem{prop}{Proposition}
\newtheorem{lemma}{Lemma}
\newtheorem{remark}{Remark}
\newtheorem{fact}{Fact}
\numberwithin{equation}{section}

\DeclareMathOperator*{\argmin}{arg\,min}

\DeclarePairedDelimiterX{\inp}[2]{\langle}{\rangle}{#1, #2}

\usepackage{hyperref}
\hypersetup{
    colorlinks=true,
    linkcolor=blue,
    citecolor=blue,
    filecolor=magenta,      
    urlcolor=blue,
}

\title{An Inexact Frank-Wolfe Algorithm for Composite Convex Optimization Involving a Self-Concordant Function}

\author{Nimita Shinde\thanks{Department of Industrial and Systems Engineering, Lehigh University, Bethlehem, PA 18015 (nis623@lehigh.edu)}
\and Vishnu Narayanan\thanks{Industrial Engineering and Operations Research, Indian Institute of Technology Bombay, Powai, Mumbai, 400076 India (vishnu@iitb.ac.in)}
\and James Saunderson\thanks{Department of Electrical and Computer Systems Engineering, Monash University, Clayton, VIC 3800, Australia (james.saunderson@monash.edu)}}

\date{}

\begin{document}

\maketitle

\begin{abstract}
In this paper, we consider Frank-Wolfe-based algorithms for composite convex optimization problems with objective involving a logarithmically-homogeneous, self-concordant functions. Recent Frank-Wolfe-based methods for this class of problems assume an oracle that returns exact solutions of a linearized subproblem. We relax this assumption and propose a variant of the Frank-Wolfe method with inexact oracle for this class of problems. We show that our inexact variant enjoys similar convergence guarantees to the exact case, while allowing considerably more flexibility in approximately solving the linearized subproblem. In particular, our approach can be applied if the subproblem can be solved prespecified additive error or to prespecified relative error (even though the optimal value of the subproblem may not be uniformly bounded). Furthermore, our approach can also handle the situation where the subproblem is solved via a randomized algorithm that fails with positive probability. Our inexact oracle model is motivated by certain large-scale semidefinite programs where the subproblem reduces to computing an extreme eigenvalue-eigenvector pair, and we demonstrate the practical performance of our algorithm with numerical experiments on problems of this form.

\end{abstract}

\section{Introduction}
Consider a composite optimization problem
\begin{equation}\tag{SC-OPT}\label{eqn:SC-gendef}
\min_{x\in \mathbb{R}^n} F(x) \equiv  \min_{x\in \mathbb{R}^n} f(\mathcal{B}(x)) + g(x),
\end{equation}
where $\mathcal{B}: \mathbb{R}^n \rightarrow \mathbb{R}^d$ is a linear map, $f$ is a convex, $\theta$-logarithmically-homogeneous, $M$-self-concordant barrier for a proper cone $\mathcal{K}$, and 
$g$ is a proper, closed, convex and possibly nonsmooth function with compact domain. We further assume that 
$\textup{int}(\mathcal{K})\cap \mathcal{B}(\textup{dom}(g)) \neq \emptyset$, ensuring that there exists at least one feasible point for~\eqref{eqn:SC-gendef}. Self-concordant functions are thrice differentiable convex functions whose third directional derivatives are bounded by their second directional derivatives (see Section~\ref{sec:Prelim} for the exact definition and key properties of these functions).
Self-concordant functions appear in several applications such as machine learning, and image processing. Several second-order~\cite{tran2015composite, dinh2013proximal, tran2014inexact,liu2021newton} as well as first-order~\cite{FW,dvurechensky2020self,dvurechensky2022generalized,zhao2022analysis} methods have been used to solve~\eqref{eqn:SC-gendef} (see Section~\ref{sec:Intro-Review} for more details).

For large-scale applications of~\eqref{eqn:SC-gendef}, one natural algorithmic approach is to use the Frank-Wolfe algorithm~\cite{frank1956algorithm}, which proceeds by partially linearizing the objective at the current iterate $x^t$, solving the linearized problem
\begin{equation}\label{prob:LinearMinimizationProblem}\tag{Lin-OPT}
\min_{h\in \mathbb{R}^n} F^{\textup{lin}}_{x^t}(h) \equiv \min_{h\in \mathbb{R}^n} \langle \mathcal{B}^*(\nabla f(\mathcal{B}(x^t))), h\rangle + g(h),
\end{equation}
and taking the next iterate to be a convex combination of $x^t$ and an optimal point for~\eqref{prob:LinearMinimizationProblem}. 
It is challenging to analyse the convergence of such first-order methods for self-concordant functions because their gradients are generally unbounded on their domains, and so  are not globally Lipschitz continuous. Despite this,  when $f$ is $\theta$-logarithmically homogeneous and self-concordant, Zhao and Freund~\cite{zhao2022analysis} established convergence guarantees for the Frank-Wolfe method when the step direction and step size are carefully chosen in a way that depends on the optimal solution of~\eqref{prob:LinearMinimizationProblem} (see 
Algorithm~\ref{algo:GFW}). This limits the applicability of the method in situations where it may not be possible to solve~\eqref{prob:LinearMinimizationProblem} to high accuracy under reasonable time and storage constraints.

In this paper, we extend the generalized Frank-Wolfe method of~\cite{zhao2022analysis} to allow for inexact solutions to the linearized subproblem~\eqref{prob:LinearMinimizationProblem} at each iteration. 
Our inexact oracle model is flexible enough to allow for situations (a) where the subproblem is only solved to a  specified relative error, and (b) where the method for solving the subproblem is randomized and fails with positive probability. Despite this flexible oracle model, we establish convergence guarantees that are comparable to those given by~\citet{zhao2022analysis} for the method with exact oracle. While the high-level strategy of the analysis is broadly inspired by that of~\cite{zhao2022analysis}, our arguments are somewhat more involved.

\begin{algorithm}[h]
\caption{Generalized Frank-Wolfe Algorithm}
\label{algo:GFW}
\textbf{Input}: Problem~\eqref{eqn:SC-gendef} where $f(\cdot)$ is a 2-self-concordant, $\theta$-logarithmically-homogeneous barrier function
\begin{algorithmic}[1] 
\STATE Initialize $x^0 \in \mathbb{R}^n$.
\FOR{$t = 0,1,\dotsc$}
\STATE \textbf{Compute update direction $h^t$.} Compute $\nabla f(\mathcal{B}(x^t))$ and $h^t \in \argmin_{h\in \mathbb{R}^n} F^{\textup{lin}}_{x^t}(h)$.
\STATE \textbf{Compute duality gap $G_t$, and distance between $h^t$ and $x^t$.} Compute $G_t = F^{\textup{lin}}_{x^t}(x^t) - F^{\textup{lin}}_{x^t}(h^t)$ and $D_t = \|\mathcal{B}(h^t - x^t)\|_{\mathcal{B}(x^t)}$.
\STATE \textbf{Compute step-size $\gamma_t$.} Set $\gamma_t = \min \left\{ \frac{G_t}{D_t(D_t+G_t)},1\right\}$.
\STATE Update $x^{t+1} = (1-\gamma_t)x^t + \gamma_t h^t$.
\ENDFOR
\end{algorithmic}
\end{algorithm}

\subsection{Questions addressed by this work}\label{sec:Intro-Issues}
In this section, we discuss three ideas to further generalize the generalized Frank-Wolfe method~\cite{zhao2022analysis}. 
These modifications help in addressing the following questions and we discuss them in detail in this section.
\begin{enumerate}
\item What if we can only generate an approximate minmizer to the linear minimizer problem~\eqref{prob:LinearMinimizationProblem}, i.e., we can only solve the problem in Step 3 of Algorithm~\ref{algo:GFW} to some given additive error?
\item What if we only have access to a method to solve~\eqref{prob:LinearMinimizationProblem}, that fails, with some nonzero probability, to return a point with specified approximation error?
\item What if we only have access to a method that can solve~\eqref{prob:LinearMinimizationProblem} with some specified relative (rather than additive) error?
\end{enumerate}

\paragraph{Motivating example.} A key motivating example for the use of an inexact 
oracle in the  probabilistic setting is the case where $g(\cdot)$ is the indicator function for the set
\begin{equation}\label{eqn:TraceSetS}
\mathcal{S}=\{X\in \mathbb{S}^n:\textup{X}\succeq 0, \textup{Tr}(X) = 1\},
\end{equation} whose extreme points are the rank-1 PSD matrices. The problem of optimizing convex functions over the set $\mathcal{S}$ prominently features in the study of memory-efficient techniques for semidefinite programming (see, e.g.,~\cite{shinde2021memory,yurtsever2019scalable}).  For such problems, generating an optimal solution to~\eqref{prob:LinearMinimizationProblem} is equivalent to determining the smallest eigenvalue-eigenvector pair of $\mathcal{B}^*(\nabla f(\mathcal{B}(X^t)))$. For large-scale problems, it is typically not practical to compute such an eigenvalue-eigenvector pair to a very high accuracy at each iteration.
Furthermore, scalable extreme eigenvalue-eigenvector algorithms (such as the Lanczos method with random initialization~\cite{kuczynski1992estimating}) usually only give relative error guarantees, and are only guaranteed to succeed with probability at least $1-p$ for some $p\in(0,1)$.

Furthermore, for large-scale instances, it is essential to have a memory-efficient way to represent the decision variable. Recent memory-efficient techniques such as sketching~\cite{yurtsever2017sketchy,yurtsever2019scalable,tropp2017practical,tropp2019streaming} or sampling~\cite{shinde2021memory} involve alternate low-memory representation of the decision variable and provide provable convergence guarantees of first-order methods that employ these low-memory representations. In Section~\ref{sec:SC-Barrier}, we briefly discuss how such memory-efficient techniques can be combined with our proposed method specifically when applied to a semidefinite relaxation of the problem of optimizing quadratic form over the sphere (see Section~\ref{sec:Intro-OptQuadForm}).

\paragraph{}
We now describe the setting for each question listed in this section.

\paragraph{Inexact Linear Minimization Oracle}
An inexact minimization oracle that we consider here does not need to generate an optimal solution to~\eqref{prob:LinearMinimizationProblem}, i.e., it does not need to solve the problem exactly in order to provide provable convergence guarantees for the algorithm. Before defining the oracle precisely, we look at how an inexact oracle changes the definition of quantities in Algorithm~\ref{algo:GFW}.
We first define the approximate duality gap, $G_t^a$, as
\[G_t^a(x^t,h^t) = F^{\textup{lin}}_{x^t}(x^t)-F^{\textup{lin}}_{x^t}(h^t).\] Here, $x^t$ is the value of the decision variable at iterate $t$ of our method, and $h^t$ is the approximate minimizer of~\eqref{prob:LinearMinimizationProblem}.
Note that, if $h^t$ is replaced by $h^{\star}$, an optimal solution to~\eqref{prob:LinearMinimizationProblem}, then  $G_t^a$ becomes equal to the Frank-Wolfe duality gap $G_t$~\cite{FW}. In Algorithm~\ref{algo:GFW}, $G_t$ is used to define the step length $\gamma_t$. Thus, in Algorithm~\ref{algo:GFW}, an exact oracle is needed to compute the update direction as well as the step length. 

We consider a setting where we compute the update direction by computing an approximate minimizer, $h^t$, to~\eqref{prob:LinearMinimizationProblem}, and correspondingly defining $G_t^a$ and the step length $\gamma_t$ which are functions of $h^t$. For our iterative algorithm to be a descent method, we need $G_t^a \geq 0$. Thus, this gives us the criteria that the output of inexact oracle must satisfy. Furthermore, our analysis shows that if, at iterate $t$, $G_t^a$ is large enough, we can ensure the algorithm makes progress, even if the approximation error in~\eqref{prob:LinearMinimizationProblem} is not small. Thus, we define an inexact oracle \texttt{ILMO} as follows:
\begin{definition}[\texttt{ILMO}]\label{def:ILMO}
\texttt{ILMO} is an oracle that takes $x^t$, $\nabla \mathcal{B}^*(f(\mathcal{B}(x^t)))$ (gradient of the objective function  at $x^t$), and a nonnegative quantity $\delta_t> 0$ as an input, and outputs a pair $(h^t,G_t^a)$, where $G_t^a$ is the approximate duality gap, such that $G_t^a\geq 0$, and at least one of the following two conditions are satisfied:
\begin{enumerate}[leftmargin=0.7cm, label=\textbf{C.\arabic*},ref=C.\arabic*]
\item \label{Cond1} (large gap condition)  $G_t^a > \theta +R_g$, where $R_g$ is the maximum variation in $g(\cdot)$ over its domain,
\item \label{Cond2} ($\delta_t$-suboptimality condition)
$F^{\textup{lin}}_{x^t}(h^t) \leq F^{\textup{lin}}_{x^t}(h^{\star}) + \delta_t$,
where $h^{\star}$ is an optimal solution to~\eqref{prob:LinearMinimizationProblem}.
\end{enumerate}
\end{definition}
Note that, condition~\ref{Cond2} ($\delta_t$-suboptimality) is equivalent to stating that the output of \texttt{ILMO}, $h^t$, is an $\delta_t$-optimal solution to~\eqref{prob:LinearMinimizationProblem}, and thus, we do not need to solve~\eqref{prob:LinearMinimizationProblem} exactly. Furthermore, from our analysis, we observed that when $G_t^a$ is lower bounded by $\theta+R_h$, $h^t$ need not be a $\delta_t$-optimal solution. Thus, the definition of \texttt{ILMO} enforces that $h^t$ needs to be $\delta_t$-optimal only when $G_t^a \leq \theta+R_h$, which can lead to reduced computational complexity of solving~\eqref{prob:LinearMinimizationProblem}.
In Section~\ref{sec:AGFW}, we propose approximate Frank-Wolfe algorithm (see Algorithm~\ref{algo:AGFW}), that only requires access to \texttt{ILMO}, and furthermore, we establish convergence guarantees comparable to that of Algorithm~\ref{algo:GFW} that requires exact oracle.

\paragraph{Inexact oracle with a nonzero probability of failure.} 
We consider an inexact oracle \texttt{PLMO} that has a nonzero probability $p$ of failure as defined below.
\begin{definition}[\texttt{PLMO}]\label{def:PLMO}
 \texttt{PLMO} is an oracle that takes $x^t$, $\mathcal{B}^*(\nabla f(\mathcal{B}(x^t)))$  (the gradient of the objective function at $x^t$), and $\delta_t$ as input, and returns a pair $(h^t,G_t^a)$ such that $G_t^a \geq 0$ and with probability at least $1-p$, at least one of the two conditions~\ref{Cond1} (large gap) and~\ref{Cond2} ($\delta_t$-suboptimality) is satisfied.
\end{definition}

The key difference between \texttt{PLMO} and \texttt{ILMO} (Definition~\ref{def:ILMO}) is the nonzero probability of failure of the oracle \texttt{PLMO}. Indeed, \texttt{ILMO} is a special case of \texttt{PLMO} when $p=0$.
In Section~\ref{sec:ProbLMO}, we consider this probabilistic setting for the oracle. We provide convergence analysis of our proposed approximate generalized Frank-Wolfe algorithm, when the linear minimization oracle fails to satisfy both conditions~\ref{Cond1} (large gap) and~\ref{Cond2} ($\delta_t$-suboptimality) with probability at most $p$.


\paragraph{Implementing the inexact oracles}
Our inexact oracle model is flexible enough that it can be implemented using an algorithm for solving~\eqref{prob:LinearMinimizationProblem} to either prespecified additive error or prespecified multiplicative error.
The multiplicative error case is of interest when
$g(\cdot)$ is the indicator function of a compact
 convex set so that $g(x) = 0$ for all feasible $x$. In this case, it turns out (see Lemma~\ref{lem:ind-opt-bnd}) that the optimal value of~\eqref{prob:LinearMinimizationProblem}, $F^{\textup{lin}}_{x^t}(h^{\star})$, is negative. By an algorithm that solves~\eqref{prob:LinearMinimizationProblem} to multiplicative error $\tau_t\in [0,1)$, we mean an algorithm that returns a feasible point $h^t$ that satisfies
\begin{equation*}
F^{\textup{lin}}_{x^t}(h^{\star}) \leq F^{\textup{lin}}_{x^t}(h^t)  \leq (1-\tau_t)F^{\textup{lin}}_{x^t}(h^{\star}),
\end{equation*}
This corresponds to an additive error of $\tau_t (-F^{\textup{lin}}_{x^t}(h^\star))$. 
The multiplicative error case is particularly challenging because the gradient of a self-concordant barrier, $\nabla f(\mathcal{B}(x^t))$, grows without bound near the boundary of the domain of $f(\cdot)$, meaning that $F^{\textup{lin}}_{x^t}(h^\star)$ is not uniformly bounded. Therefore, a small relative error for~\eqref{prob:LinearMinimizationProblem} (i.e., a small value of $\tau_t$) may still lead to large absolute errors. In Section~\ref{sec:AddMultiErr} we show how the large gap condition (Condition~\ref{Cond1}) in our oracle model allows us to overcome this issue.

\subsection{Applications}\label{sec:Intro-Applications}
Consider the $M$-self-concordant, $\theta$-logarithmically-homogeneous barrier function
\begin{equation}\label{eqn:SC-Type1}
f(\mathcal{B}(x)) = -\sum_{i=1}^m c_i\log(a_i^Tx),
\end{equation}
with $M=\max_i \frac{2}{\sqrt{c_i}}$ and $\theta = \sum_{i=1}^m c_i$, where the linear map $\mathcal{B}(\cdot): \mathbb{R}^n\rightarrow \mathbb{R}^m$, maps the decision variable $x\in \mathbb{R}^n$ to an $m$-dimensional vector $[a_i^Tx]_{i=1}^m$.
In this section, we review two applications (see Subsections~\ref{sec:Intro-OptQuadForm} and~\ref{sec:Intro-MLQST}) whose convex formulations involve minimizing functions of the form~\eqref{eqn:SC-Type1} over $\mathcal{S}$~\eqref{eqn:TraceSetS}. The problems of this type are also often studied in the context of memory-efficient algorithms. One of the contributions of this paper is to show that, for these applications,  we can combine memory-efficient representations~\cite{shinde2021memory} of the decision variable with the first-order algorithm proposed in this paper.

Historically, self-concordant barrier functions were often studied because they allow for an affine-invariant analysis of Newton's method~\cite{nesterov1994interior}. Recently, problems with self-concordant barriers in the objective are also studied in the context of machine learning, often as a loss function~\cite{bach2010self,owen2013self,marteau2019beyond}.
 The function~\eqref{eqn:SC-Type1} appears in several applications such as maximum-likelihood quantum state tomography~\cite{hradil1997quantum}, portfolio optimization~\cite{vardi1993image,cover1984algorithm,zhao2021non}, and maximum likelihood Poisson inverse problems~\cite{harmany2011spiral,nowak1998multiscale,nowak2000multiscale} in a variety of imaging applications such as night vision, low-light imaging.
Furthermore, convex formulations of problems such as $D$-optimal design~\cite{de1995d,atwood1969optimal} and learning sparse Gaussian Markov Random Fields~\cite{dinh2013proximal,rue2005gaussian} also involve optimizing a 2-self-concordant, $n$-logarithmically-homogeneous barrier function ($f(X) = -\log(\det(X))$, with $X\in \mathbb{S}^n_{++}$). The functions $-\log \det X$ and $-\log x$ are also commonly used self-concordant barriers for the PSD cone and $\mathbb{R}_+$, respectively. 


\subsubsection{Optimizing quadratic form over the sphere} \label{sec:Intro-OptQuadForm}
Consider the problem
\begin{equation}\label{prob:QuadSphere}\tag{GMean}
\max_{\|x\|=1,x\in \mathbb{R}^n} \prod_{i=1}^d \langle xA_i,x\rangle^{1/d},
\end{equation}
where $A_i\in \mathbb{S}^n_+$ for $i=1,\dotsc,d$. Many applications such as approximating the permanent of PSD matrices or  
portfolio optimization take the form of maximizing quadratic forms $\prod_{i=1}^d \langle xA_i,x\rangle$ over $ \{x\in \mathbb{R}^n:\|x\|=1\}$. By normalizing the objective of this problem, we get~\eqref{prob:QuadSphere}, whose objective is the geometric mean of the quadratic forms. \citet[Theorem~7.1]{yuan2021semidefinite} shows that when $d = \Omega(n)$,~\eqref{prob:QuadSphere} is NP-hard.
A natural semidefinite relaxation of~\eqref{prob:QuadSphere} is
\begin{equation}\label{prob:SCQuadProdSphere}
\max_X \prod_{i=1}^d \langle A_i,X\rangle^{1/d}\quad \textup{subject to}\ 
\begin{cases}
\textup{Tr}(X) = 1,\\
X\succeq 0.
\end{cases}
\end{equation}
\citet{yuan2021semidefinite} show that solving the SDP relaxation~\eqref{prob:SCQuadProdSphere} and appropriately rounding the solution to have rank one, gives a constant factor approximation algorithm for~\eqref{prob:QuadSphere}. For $\mathcal{S}$ given in~\eqref{eqn:TraceSetS}, we have an equivalent problem
\begin{equation}\label{prob:SC-SumLogDef}\tag{SC-GMean}
\min_{X\in \mathcal{S}} -\sum_{i=1}^d \log(\langle A_i,X\rangle),
\end{equation}
where the objective is a 2-self-concordant, $d$-logarithmically-homogeneous barrier function, and~\eqref{prob:SC-SumLogDef} is equivalent to~\eqref{prob:SCQuadProdSphere} up to the scaling of the objective.
So, an optimal solution of~\eqref{prob:SC-SumLogDef} is also an optimal solution of~\eqref{prob:SCQuadProdSphere}.
\citet[Theorem~1.3]{yuan2021semidefinite} show that if $X^\star$ is optimal for~\eqref{prob:SCQuadProdSphere} and $y\sim \mathcal{N}(0,X^\star)$ then $y/\|y\|$ feasible for~\eqref{prob:QuadSphere} and within a factor of $e^{-L_r}$ of the optimal value, where $L_r = \gamma + \log 2 + \psi \left(\frac{r}{2}\right) - \log \left(\frac{r}{2}\right) < 1.271$, $\gamma$ is the Euler-Mascheroni constant, $\psi(r)$ is the digamma function.


\subsubsection{Maximum-likelihood Quantum State Tomography (ML-QST)}\label{sec:Intro-MLQST}
The problem with structure similar to~\eqref{prob:SC-SumLogDef} also appears in Quantum State Tomography (QST). 
QST is the process of estimating an unknown quantum state from given measurement data. We consider that a measurement is a random variable $z$ which takes the values in the set $[N] = \{1,\dotsc,N\}$ with probability distribution
$P(z = i) = \langle B_i,X\rangle, \forall i \in [N]$,
where each $B_i$ is PSD complex matrix such that $\sum_{i=1}^N B_i = I$, $X\in \mathcal{S}$, where $\mathcal{S}$ is defined in~\eqref{eqn:TraceSetS}. Note that, if there are $q$ qubits, then $n = 2^q$.
Maximum-likelihood estimation of quantum states~\cite{hradil1997quantum,hradil20043} involves finding the state (represented by $X\in \mathcal{S}$) that maximizes the probability of observing the measurement. Assume that we have $m$ measurements $z_1,\ldots,z_m$ where $z_j\in [N]$ for $j=1,\ldots,m$. The maximum-likelihood QST problem then takes the form
$\max_{X\in \mathcal{S}} \prod_{j=1}^m \langle A_j, X\rangle$,
where $A_j = B_{z_j}$ for $j=1,\dotsc, m$.
Several 
results for maximum likelihood estimation of QST are provided in literature (\cite{glancy2012gradient,baumgratz2013scalable,scholten2018behavior}).
The problem can be
equivalently stated as
\begin{equation}\label{prob:MLQST}\tag{ML-QST}
\min_{X\in \mathcal{S}} -\sum_{j=1}^m \log(\langle A_j,X\rangle).
\end{equation}
Note that~\eqref{prob:MLQST} is identical  to~\eqref{prob:SC-SumLogDef} and its objective is a 2-self-concordant, $m$-logarithmically-homogeneous barrier.

\subsection{Contributions} 
The focus of this paper is to provide a framework to solve a minimization problem of the form~\eqref{eqn:SC-gendef}, where the objective function is a self-concordant, logarithmically-homogeneous barrier function. 
We provide an algorithmic framework that addresses the three questions stated in Section~\ref{sec:Intro-Issues}. The main contributions of the paper are summarized below.
\begin{itemize}
\item We provide an algorithmic framework, that generates an $\epsilon$-optimal solution to~\eqref{eqn:SC-gendef}. The proposed algorithm (Algorithm~\ref{algo:AGFW}) only requires access to an inexact oracle \texttt{ILMO} (Definition~\ref{def:ILMO}). 
We provide provable convergence guarantees (Lemmas~\ref{lemma:ConvDet-OL} and~\ref{lemma:ConvDet-CL}) that are comparable to guarantees given by~\citet{zhao2022analysis} whose method requires access to an exact oracle.

\item We consider the setting where the inexact oracle can provide an approximate solution to~\eqref{prob:LinearMinimizationProblem} with either additive or multiplicative error bound. We show how to implement the oracle that satisfies the output criteria of \texttt{ILMO}. In this scenario, the convergence guarantees given by Lemmas~\ref{lemma:ConvDet-OL} and~\ref{lemma:ConvDet-CL} for Algorithm~\ref{algo:AGFW} still hold true.


\item We next consider a setting where the linear minimization oracle fails to satisfy both conditions~\ref{Cond1} (large gap) and~\ref{Cond2} ($\delta_t$-suboptimality) with probability at most $p>0$. In this setting, we show that we can bound, with probability at least $1-p$, the number of iterations of Algorithm~\ref{algo:AGFW} required to converge to an $\epsilon$-optimal solution of~\eqref{eqn:SC-gendef} (see Lemmas~\ref{lemma:ConvProb-OL} and~\ref{lemma:ConvProb-CL}).

\item We apply our algorithmic framework specifically to~\eqref{prob:SC-SumLogDef} and provide the computational complexity of our proposed method. We also show that our method can be combined with a memory-efficient technique given by~\citet{shinde2021memory}, so that the method uses at most $\mathcal{O}(n+d)$ memory (apart from the memory used to store the input parameters). We also implement our method as well as Algorithm~\ref{algo:GFW}, and use these algorithms to generate an $\epsilon$-optimal solution to randomly generated instances of~\eqref{prob:SC-SumLogDef}. We show that, for these instances, the behavior of our method (in terms of change in optimality gap with each iteration) is similar to that of Algorithm~\ref{algo:GFW}. We also show that the theoretical upper bound on the number of iterations to converge to an $\epsilon$-optimal solution is quite conservative, and in practice, the algorithms generate the desired solution much quicker.

\end{itemize}

\subsection{Brief Review of First- and Second-Order Methods for~\eqref{eqn:SC-gendef}}\label{sec:Intro-Review}
It is easy to see that interior-point methods~\cite{nesterov2003introductory, nemirovski2008interior} can be used to solve composite problems of type~\eqref{eqn:SC-gendef}. In fact, a standard technique while using IPMs for such problems first eliminates the nonsmooth term $h(x)$ and additionally uses (self-concordant) barrier functions for inequality constraints. The method then solves a sequence of problems to generate an approximate solution to~\eqref{eqn:SC-gendef}. 
Such methods often result in the loss of underlying structures (for example, sparsity) of the input problem.
To overcome this issue, proximal gradient-based methods~\cite{tran2015composite, dinh2013proximal, tran2014inexact} directly solve the nonsmooth problem~\eqref{eqn:SC-gendef} and provide local convergence guarantees ($\|x^{t+1}-x^t\|_{x^t}\leq \epsilon$, where $\|u\|_{x^t} = \sqrt{u^T\nabla^2f(x^t)u}$) for~\eqref{eqn:SC-gendef}. \citet{liu2021newton} use projected Newton method to solve $\min_{x\in\mathcal{C}} f(x)$, where $f(\cdot)$ is self-concordant, such that the projected Newton direction is computed in each iteration of the method using an adaptive Frank-Wolfe method to solve the subproblem. They show that the outer loop (projected Newton) and the inner loop (Frank-Wolfe) have the complexity of $\mathcal{O}(\log (1/\epsilon))$ and $\mathcal{O}(1/\epsilon^\nu)$ respectively, where $\nu = 1+\frac{\log(1-2\beta)}{\log(\sigma)}$, for some constants $\sigma, \beta > 0$, where $(\beta,\sigma)$ satisfy the conditions outlined in~\cite[Theorem~4.2]{liu2021newton}.

For large-scale applications of~\eqref{eqn:SC-gendef}, Frank-Wolfe and other first-order methods are often preferred since, at each iteration either an exact or approximate solution is generated to the linear minimization problem~\eqref{prob:LinearMinimizationProblem}
which can lead to low-computational complexity per iteration. However, the analysis of such methods usually depends on the assumption that $f$ has a Lipschitz continuous gradient~\cite{FW}, which a self-concordant function does not satisfy.  There are algorithms (for example \cite{dvurechensky2020self}) that provide new policies to compute step length at each iteration of Frank-Wolfe and provide convergence guarantees for the modified algorithm. However, these guarantees depend on the Lipschitz constant of the objective on a sublevel set, which might not be easy to compute and in some cases may not even be bounded. 

\subsection{Notations and Outline}
\paragraph{Notations.}
We use small alphabets (for instance, $x$) to denote a vector variables, and big alphabets (for example, $X$) to denote matrix variables. We use superscripts $(x^{(t)})$ to denote values of parameters and variables at $t$-th iteration of the algorithm.
The term $\left\langle A, B \right\rangle = \textup{Tr}\left(A^TB\right)$ denotes the matrix inner product.
The notations $\nabla g(\cdot)$ and $\nabla^2g(\cdot)$ are used to denote the gradient and the Hessian respectively of a twice differentiable function $g$. The linear map $\mathcal{B}(\cdot):\mathbb{S}^n\rightarrow \mathbb{R}^d$ maps a symmetric $n\times n$ matrix to a $d$-dimensional space and $\mathcal{B}^*(\cdot)$ is its adjoint.
The notation $\mathcal{O}$ has the usual complexity interpretation.

\paragraph{Outline.}
In Section~\ref{sec:Prelim}, we review the generalized Frank-Wolfe algorithm and key properties of a 2-self-concordant, $\theta$-logarithmically-homogeneous function. We also define the terminology used in the rest of the paper.
In Section~\ref{sec:AGFW}, we propose an approximate generalized Frank-Wolfe algorithm (Algorithm~\ref{algo:AGFW}) that generates an $\epsilon$-optimal solution to~\eqref{eqn:SC-gendef}, and we provide convergence analysis of the algorithm.
In Section~\ref{sec:AddMultiErr}, we show how an oracle that generates an approximate solution with relative error can be used with our proposed method (Algorithm~\ref{algo:AGFW}).
In Section~\ref{sec:ProbLMO}, we propose a method (Algorithm~\ref{algo:AGFW}) that, with high probability, generates an $\epsilon$-optimal solution to~\eqref{eqn:SC-gendef} when the oracle that solves the linear subproblem has a nonzero probability of failure.
In Section~\ref{sec:SC-Barrier}, we apply our method specifically to~\eqref{prob:SC-SumLogDef} and provide the convergence analysis of the method. We also show that our method can be combined with memory-efficient techniques from literature so that it can implemented in $\mathcal{O}(n+d)$ memory.
Finally, in Section~\ref{sec:NE}, we provide computational results of applying Algorithm~\ref{algo:AGFW} to random instances of~\eqref{prob:SC-SumLogDef}.

\section{Preliminaries}\label{sec:Prelim}
In this section, we review key properties of self-concordant, $\theta$-logarithmically-homogeneous barrier functions as well as other quantities that are useful in our analysis. For a detailed review of self-concordant functions, refer to~\citet{nesterov1994interior}.
\begin{definition}\cite{nesterov1994interior}\label{def:SCFunction}
Let $\mathcal{K}$ be a regular cone, and
let $f$ be a thrice-continuously differentiable, strictly convex function defined on $\textup{int}(\mathcal{K})$. Define 
the third-order directional derivative $\nabla^3 f(w)[u_1,u_2,u_3] = \frac{\partial^3}{\partial t_1 \partial t_2 \partial t_3} f(w+t_1u_1 + t_2u_2 + t_3u_3)|_{t_1=t_2=t_3 = 0}$. Then $f(w)$ is a $M$-self-concordant, $\theta$-logarithmically-homogeneous barrier on  $\textup{int}(\mathcal{K})$ if
\begin{enumerate}
\item $f(w_i)\rightarrow \infty$ along every sequence $\{w_i\} \in \textup{int}(\mathcal{K})$ that converges to a boundary point of $\mathcal{K}$,
\item $\left|\nabla^3 f(w)[u,u,u]\right| \leq M\left( u^T\nabla^2 f(w)u\right)^{3/2}$ for all $w\in \textup{int}(\mathcal{K})$ and $u\in \mathbb{R}^d$, and
\item $f(cw) = f(w) - \theta \log(c)$ for all $w\in \textup{int}(\mathcal{K})$ and $c > 0$.
\end{enumerate}
\end{definition}

\paragraph{Properties of a $2$-self concordant, $\theta$-logarithmically homogeneous barrier function $f$.} 
Let $\mathcal{K}$ be a proper cone and let $f(\cdot)$ be a $2$-self-concordant, $\theta$-logarithmically-homogeneous barrier function on $\textup{int}(\mathcal{K})$.
For any $u\in \textup{int}(\mathcal{K})$, let $\mathcal{H}(u)$ denote the Hessian  of $f(\cdot)$ at $u$. Define the (semi)norm $\|w\|_u = \sqrt{\langle w,\mathcal{H}(u)w\rangle}$ for any $w\in \mathbb{R}^d$. Also, define $\mathcal{D}(u,1) = \{w\in \mathcal{K}:\|w-u\|_u < 1\}$ which is often referred to as the unit Dikin ball. Define a univariate function $\omega$ and its Fenchel conjugate as follows:
\begin{align}
\omega(a) &= \begin{cases}-a - \log(1-a) & \textup{if $a<1$}\\\infty & \textup{otherwise}.\end{cases}\\
\omega^{\star}(a) &= \begin{cases} a - \log(1+a) & \textup{if $a>-1$}\\ \infty & \textup{otherwise.}\end{cases}
\end{align}

\begin{prop}[Proposition~2.3.4, Corollary~2.3.1,2.3.3~\cite{nesterov1994interior}]\label{prop:SCLHProperties}
Let $\mathcal{K}$ be a proper cone and let $f(\cdot)$ be a $2$-self-concordant, $\theta$-logarithmically-homogeneous barrier function on $\textup{int}(\mathcal{K})$. If $u \in \textup{int}(\mathcal{K})$ then 
\begin{align}
\label{eqn:SCSOCCond} f(w) - f(u) -\langle\nabla f(u), w-u\rangle &\leq \omega(\|w-u\|_u)\quad \forall w\in \mathcal{D}(u,1)\\
\label{eqn:SCProp1}|\langle\nabla f(u), w\rangle|&\leq \sqrt{\theta}\|w\|_u, \quad \quad\;\,\forall w\in \mathbb{R}^d\\
\label{eqn:SCProp2}\|w\|_u &\leq -\langle \nabla f(u),w\rangle\quad \forall w\in \mathcal{K}\\
\label{eqn:SCProp3}\langle \nabla f(u),w\rangle &= -\langle \mathcal{H}(u)u,w\rangle \quad \forall w\in \mathbb{R}^d\\
\label{eqn:SCProp4}\langle \nabla f(u),u\rangle &= -\theta\quad \textup{and}\\
\label{eqn:SCProp5}\theta &\geq 1.
\end{align}
\end{prop}
These properties will be used later in the proofs.

\subsection{Terminology}\label{sec:terminology}
We now define a few quantities that are used in the rest of the paper. 
\begin{itemize}
\item Frank-Wolfe duality gap, $G_t(x^t)$~\cite{FW}: For any feasible point $x^t$ to~\eqref{eqn:SC-gendef}, let $h^{\star}$ be an optimal solution to~\eqref{prob:LinearMinimizationProblem}. Then, 
\begin{equation}\label{eqn:DefGk}
G_t(x^t) = \langle \nabla f(\mathcal{B}(x^t)),\mathcal{B}(x^t- h^{\star})\rangle + g(x^t) - g(h^{\star}) = F^{\textup{lin}}_{x^t}(x^t) - F^{\textup{lin}}_{x^t}(h^\star).
\end{equation}
\item Approximate Frank-Wolfe duality gap, $G_t^a(x^t,h^t)$: For any two feasible points $x^t$, $h^t$ to~\eqref{eqn:SC-gendef}, we have
\begin{equation}\label{eqn:DefGkApprox}
G_t^a(x^t,h^t) = \langle \nabla f(\mathcal{B}(x^t)),\mathcal{B}(x^t- h^{t})\rangle + g(x^t) - g(h^{t})= F^{\textup{lin}}_{x^t}(x^t) - F^{\textup{lin}}_{x^t}(h^t).
\end{equation}
\textbf{Note:} The duality gap, $G_t(x^t)$, is a function of the current iterate $x^t$, whereas the approximate duality gap, $G_t^a(x^t,h^t)$, is a function of the current iterate $x^t$, and any feasible point $h^t$ to~\eqref{eqn:SC-gendef}. However, for the ease of notation, we denote $G_t(x^t)$ by $G_t$, and $G_t^a(x^t,h^t)$ by $G_t^a$ in the rest of the paper when $x^t$ and $h^t$ are clear from the context.
\item Optimality gap $\Delta_t$: If $x^{\star} = \argmin_{x\in \mathbb{R}^n}F(x)$, i.e., $x^{\star}$ is an optimal solution to~\eqref{eqn:SC-gendef}, and $x^t$ is a feasible point to~\eqref{eqn:SC-gendef} then
\begin{equation}\label{eqn:RelationdeltaGk}
\Delta_t(x^t) = F(x^t) - F(x^{\star}).
\end{equation}
When $x^t$ is clear from the context we denote $\Delta_t(x^t)$ by $\Delta_t$.
\item Maximum variation of $g(\cdot)$ over its domain, $R_g$: We define as $R_g$ as
\begin{equation}\label{eqn:DefRh}
R_g = \max_{x_1,x_2\in \textup{dom}(g)} |g(x_1)-g(x_2)|.
\end{equation}
\item Distance, $D_t$: For any two feasible points $x^t$ and $h^t$ to~\eqref{eqn:SC-gendef}, we define 
\begin{equation}\label{eqn:DefDt}
D_t(x^t,h^t) = \|\mathcal{B}(h^t-x^t)\|_{\mathcal{B}(x^t)},
\end{equation}
where $\|\cdot\|_{\mathcal{B}(x^t)}$ is the (semi)norm defined by the Hessian of $f(\cdot)$ at $\mathcal{B}(x^t)$. When $x^t$ and $h^t$ are clear from the context, we denote $D_t(x^t,h^t)$ by $D_t$.
\end{itemize}
The following standard result shows that the optimality gap is bounded by the duality gap.
\begin{fact}
\label{fct:gap-bnd}
For any feasible point $x^t$ to~\eqref{eqn:SC-gendef}, we have $\Delta_t \leq G_t$.
\end{fact}
\begin{proof}
Since $f$ is convex, for any $x^t$ and $w$ that are feasible for~\eqref{eqn:SC-gendef}, 
\[ 0 \leq f(\mathcal{B}(w)) - \left[f(\mathcal{B}(x^t)) - \langle \nabla f(\mathcal{B}(x^t)), \mathcal{B}(w-x^t)\rangle\right].\]
Setting $w$ to be $x^\star$, an optimal point for~\eqref{eqn:SC-gendef}, and rearranging gives
\begin{align*}
    \Delta_t = f(\mathcal{B}(x^t)) - f(\mathcal{B}(x^\star)) + g(x^t) - g(x^\star) & \leq \langle \nabla f(\mathcal{B}(x^t)),\mathcal{B}(x^t-x^\star)\rangle + g(x^t) - g(x^\star).
    \end{align*}
Finally, if $h^\star$ is optimal for~\eqref{prob:LinearMinimizationProblem} 
then 
\begin{equation*}
    \langle \nabla f(\mathcal{B}(x^t)),\mathcal{B}(x^t-x^\star)\rangle + g(x^t) - g(x^\star)
     \leq \langle \nabla f(\mathcal{B}(x^t)),\mathcal{B}(x^t-h^\star)\rangle + g(x^t) - g(h^\star) = G_t,
\end{equation*}
completing the argument.
\end{proof}

\section{Approximate Generalized Frank-Wolfe algorithm}\label{sec:AGFW}
Consider the composite optimization problem~\eqref{eqn:SC-gendef},
where $f(\cdot)$ is a $\theta$-logarithmically-homogeneous, $M$-self-concordant barrier function on $\textup{int}(\mathcal{K})$, $\mathcal{K}$ is a proper cone, and $g(\cdot)$ is a proper, closed, convex and possibly nonsmooth function with $\textup{dom}(g)$ a closed, compact, nonempty set.
\citet{zhao2022analysis} propose a Generalized Frank-Wolfe algorithm (Algorithm~\ref{algo:GFW}) that requires an exact solution to the subproblem~\eqref{prob:LinearMinimizationProblem} at each iteration. In their algorithm, the knowledge of the exact minimizer is necessary to compute the update step, $h^{\star}$, as well as the step length (since the step length depends on the duality gap $G_t(x^t)$). Furthermore, the analysis of the algorithm also depends on the exact minimizer of the subproblem~\eqref{prob:LinearMinimizationProblem}. 
We, however, propose an `approximate' generalized Frank-Wolfe algorithm.
The term approximate in the name of our algorithm comes from the fact that we do not require the linear minimization subproblem~\eqref{prob:LinearMinimizationProblem} to be solved exactly at each iteration. The framework of our method is given in Algorithm~\ref{algo:AGFW}. We provide a detailed explanation of Steps 5 and 6 (setting the value of $\delta_t$ and \texttt{ILMO} respectively) below.

\begin{algorithm}[tbh]
\caption{Approximate Generalized Frank-Wolfe}
\label{algo:AGFW}
\textbf{Input}: Problem~\eqref{eqn:SC-gendef}, where $f(\cdot)$ is a 2-self-concordant, $\theta$-logarithmically-homogeneous barrier function, $\epsilon\leq \theta+R_g$: suboptimality parameter, scheduled strategy~\ref{Strat1} or adaptive strategy~\ref{Strat2} for setting the value of $\delta_t$\\
\textbf{Output}: $\widehat{x}_{\epsilon}$, such that $\widehat{x}_{\epsilon}$ an $\epsilon$-optimal solution to~\eqref{eqn:SC-gendef}
\begin{algorithmic}[1] 
\STATE Initialize $x^0 \in \mathbb{R}^n$.
\STATE Set $t\leftarrow 0$.
\WHILE{true}
\STATE Compute $\mathcal{B}^{*}(\nabla f(\mathcal{B}(x^t)))$.
\STATE Set $\delta_t \geq 0$ according to the input strategy.
\STATE $(h^t,G_t^a) = \texttt{ILMO} (\mathcal{B}^{*}(\nabla f(\mathcal{B}(x^t))),\delta_t, x^t)$.
\IF {$G_t^a \leq \epsilon$ and $\delta_t \leq 3\epsilon/2$} \STATE \textbf{return $(x_t,\delta_t)$}.\ENDIF
\STATE Compute $D_t = \|\mathcal{B}(h^t-x^t)\|_{\mathcal{B}(x^t)}$.
\STATE Set $\gamma_t = \min \left\{ \frac{G_t^a}{D_t(D_t+G_t^a)},1\right\}$.
\STATE Update $x^{t+1} = (1-\gamma_t)x^t + \gamma_t h^t$.
\STATE $t = t+1$.
\ENDWHILE
\end{algorithmic}
\end{algorithm}

Note that any $M$-self concordant function can be converted to a $2$-self concordant function $f(u)$ by scaling the objective. We define this property more formally in the following proposition.

\begin{prop}\label{prop:GenSC-2SC}
Let $\overline{f}(u)$ be an $M$-self concordant, $\overline{\theta}$-logarithmically homogeneous function. Define a function $f(u)$ and parameter $\theta$ such that
\begin{enumerate}
\item $f(u) = \frac{M^2}{4}\overline{f}(u)$,
\item $\theta= \frac{M^2}{4}\overline{\theta}$.
\end{enumerate}
Then $f(u)$ is a $2$-self concordant, $\theta$-logarithmically homogeneous function. 
\end{prop}

Proposition~\ref{prop:GenSC-2SC} tells us that, by rescaling, we can always reduce to the regime of functions that are 2-self-concordant (often called standard self-concordant). However, if we are minimizing such a function, rescaling affects the optimal value, and hence affects approximation error guarantees.
Thus, even though the input to Algorithm~\ref{algo:AGFW}  is specifically a problem with 2-self-concordant barrier function in the objective, by replacing $\epsilon$ by $\frac{M^2}{4}\epsilon$ in Step 7 of the algorithm, we can generate an $\epsilon$-optimal solution to~\eqref{eqn:SC-gendef}, when $f(\cdot)$ is an $M$-self-concordant, $\theta$-logarithmically-homogeneous barrier function.

\paragraph{Choosing $\delta_t$.}
We propose two different strategies to choose the value of $\delta_t$ in Step 5 of the algorithm.
\begin{enumerate}[leftmargin=0.9cm, label=\textbf{S.\arabic*},ref=S.\arabic*]
\item \label{Strat1} (Scheduled strategy) $(\delta_t)_{t\geq 0}$ is any positive sequence converging to zero. It follows that, given any $\epsilon>0$, there exists some positive integer $K$ such that $\delta_t\leq \epsilon/2$ for all $t \geq K$.
\item \label{Strat2} (Adaptive strategy) $\delta_t \leq \epsilon/2+\min_{\tau<t}G_{\tau}^a\quad \forall t$.
\end{enumerate}

\begin{remark}\label{remark:Strategy2}
The main idea behind adaptive strategy~\ref{Strat2} is to allow the target accuracy of \texttt{ILMO} to 
adapt to the progress of the algorithm (in terms of the smallest approximate duality gap we have encountered so far). This allows the \texttt{ILMO} to solve to low accuracy until significant progress has been made in terms of approximate duality gap. Since we are working with the approximate duality gap, and not the actual duality gap, it is possible that $G_t^a = 0$ before the algorithm has reached the desired optimality gap. We use the additive quantity $\epsilon/2$ to ensure that we can choose $\delta_t > 0$ even when $G_{\tau}^a = 0$ for some $\tau<t$.
\end{remark}

\paragraph{\texttt{ILMO} in Algorithm~\ref{algo:AGFW}.}
The main difference between generalized Frank-Wolfe~\cite{zhao2022analysis} and Algorithm~\ref{algo:AGFW}, is the definition of the oracle. The generalized Frank-Wolfe algorithm (Algorithm~\ref{algo:GFW}) requires \texttt{LMO} to generate an optimal solution to the subproblem~\eqref{prob:LinearMinimizationProblem} in order to generate update direction and step length, as well as to prove its convergence guarantees. Whereas in Algorithm~\ref{algo:AGFW}, \texttt{ILMO} is an oracle whose output only needs to satisfy $G_t^a\geq 0$ and at least one the two conditions~\ref{Cond1} (large gap) and~\ref{Cond2} ($\delta_t$-suboptimality) (see Definition~\ref{def:ILMO}).

When condition~\ref{Cond2} ($\delta_t$-suboptimality) is satisfied by the output of \texttt{ILMO}, we have the following relationship between the duality gap $G_t$ and the approximate duality gap $G_t^a$.

\begin{lemma}[Relationship between $G_t$ and $G_t^a$]\label{lemma:relationGtGta}
Let $G_t$, $G_t^a$ be as defined in Section~\ref{sec:terminology}. If condition~\ref{Cond2} ($\delta_t$-suboptimality) is satisfied, then 
\begin{equation}
G_t - \delta_t \leq G_t^a \leq G_t.
\end{equation}
\end{lemma}

\begin{proof}
If condition~\ref{Cond2} ($\delta_t$-suboptimality) holds, then Algorithm~\ref{algo:AGFW} generates $h^t$ such that
\begin{equation}\label{eqn:approxLMO}
F^{\textup{lin}}_{x^t}(h^{\star})\leq F^{\textup{lin}}_{x^t}(h^t)= F^{\textup{lin}}_{x^t}(h^{\star}) + \delta_t.
\end{equation}
From the definition of $G_t^a$ (see~\eqref{eqn:DefGkApprox}), we have 
\begin{equation*}
\begin{split}
G_t^a &= F^{\textup{lin}}_{x^t}(x^t) - F^{\textup{lin}}_{x^t}(h^t)\\
&= F^{\textup{lin}}_{x^t}(x^t) - F^{\textup{lin}}_{x^t}(h^{\star}) - \delta_t\\
&= G_t -\delta_t,
\end{split}
\end{equation*}
where the second equality follows from the second inequality in~\eqref{eqn:approxLMO}. We also have
\begin{equation*}
G_t^a \leq F^{\textup{lin}}_{x^t}(x^t) - F^{\textup{lin}}_{x^t}(h^{\star}) = G_t,
\end{equation*}
where the inequality follows from the first inequality in~\eqref{eqn:approxLMO}. Combining the two results completes the proof.
\end{proof}

\paragraph{Computing $\gamma_t$.}
We now explain how the step-size, $\gamma_t$, is selected in Algorithm~\ref{algo:AGFW}. Substituting $x^t$ and $x^{t+1}$ (as given by Algorithm~\ref{algo:AGFW}) in~\eqref{eqn:SCSOCCond}, we have
\begin{equation}\label{eqn:ProofGammat1}
f(\mathcal{B}(x^t + \gamma_t(h^t - x^t))) \leq f(\mathcal{B}(x^t))+ \gamma_t\langle \nabla f(\mathcal{B}(x^t)), \mathcal{B}(h^t-x^t)\rangle + \omega(\gamma_tD_t).
\end{equation}
Furthermore, since $g(\cdot)$ is convex,
\begin{equation}\label{eqn:ProofGammat2}
g(x^t + \gamma_t(h^t-x^t))  = g((1-\gamma_t)x^t + \gamma_th^t) \leq (1-\gamma^t)g(x^t) + \gamma^tg(h^t) = g(x^t) - \gamma_t(g(x^t)-g(h^t)).
\end{equation}
Now, adding~\eqref{eqn:ProofGammat1} and~\eqref{eqn:ProofGammat2}, we have
\begin{equation}\label{eqn:itrIneq}
\begin{split}
F(x^t+\gamma_t(h^t-x^t)) &= f(\mathcal{B}(x^t + \gamma_t(h^t - x^t))) + g(x^t + \gamma_t(h^t-x^t))\\
&\leq F(x^t) -\gamma_t G_t^a +\omega(\gamma_tD_t),
\end{split}
\end{equation}
where the last inequality follows from~\eqref{eqn:DefGkApprox}.
We can now compute an optimal step length $\gamma_t$ at every iteration $t$ by optimizing the right hand side of~\eqref{eqn:itrIneq} for $\gamma_t$ to get
\begin{equation}\label{eqn:DefStepLengthAGFW}
\gamma_t = \min \left\{ \frac{G_t^a}{D_t(D_t+G_t^a)},1\right\}.
\end{equation}
\citet{zhao2022analysis} use a similar technique to compute $\gamma_t$ based on $G_t$. The gap $G_t$ requires the knowledge of an optimal solution to the subproblem~\eqref{prob:LinearMinimizationProblem}, whereas~\eqref{eqn:DefStepLengthAGFW} uses an approximate solution to~\eqref{prob:LinearMinimizationProblem}.

\paragraph{Behavior of $\Delta_t$.} In the following proposition, we give lower bounds on the amount by which $\Delta_t$ decreases at each iteration of Algorithm~\ref{algo:AGFW}. Note that this result is independent of the sequence $(\delta_t)_{t\geq 0}$.

\begin{prop}\label{prop:DeltaConv}
Let~\eqref{eqn:SC-gendef} (with $f(\mathcal{B}(x))$ a 2-self-concordant and $\theta$-logarithmically-homogeneous barrier function) be the input to Algorithm~\ref{algo:AGFW}. Let the quantities $R_g$ and $G_t^a$ be as defined in~\eqref{eqn:DefRh} and~\eqref{eqn:DefGkApprox} respectively. 
For any $t\geq 0$ in Algorithm~\ref{algo:AGFW},
\begin{enumerate}
\item if $G_t^a > \theta +R_g$, we have
\begin{equation}\label{eqn:CaseIKeyEqn}
\Delta_t - \Delta_{t+1} \geq \frac{1}{10.6},
\end{equation}
    \item if $G_t^a\leq \theta +R_g$, we have
\begin{equation}\label{eqn:CaseIIKeyeqn}
\Delta_t - \Delta_{t+1} \geq \frac{(G_t^{a})^2}{12(\theta + R_g)^2}.
\end{equation}
\end{enumerate}
Consequently, $(\Delta_t)_{t\geq 0}$ is a monotonically non-increasing sequence.
\end{prop}

\begin{proof}
The proof of this result is similar to the proof given in~\cite{zhao2022analysis}. The difference arises due to the slightly different value of the step length $\gamma_t$ that we have used. We give the proof here for completeness. 

\paragraph{Proving~\eqref{eqn:CaseIKeyEqn}.}
Rewriting~\eqref{eqn:DefGkApprox}, we have
\begin{equation}
  -G_t^a = F^{\textup{lin}}_{x^t}(h^t) - F^{\textup{lin}}_{x^t}(x^t) \leq 0,
\end{equation}
since $G_t^a \geq 0$ for all $t$. Thus,
\begin{equation}\label{eqn:GenFWResultPf1}
 \begin{split}
G_t^a &= |\langle \nabla f(\mathcal{B}(x^t)), \mathcal{B}(h^t-x^t)\rangle+ g(h^t) - g(x^t)|\\
&\leq |\langle \nabla f(\mathcal{B}(x^t)), \mathcal{B}(h^t-x^t)\rangle|+|g(h^t) - g(x^t)|.
\end{split}
\end{equation}
Now, using the definition of $D_t$ and~\eqref{eqn:SCProp1}, we can write
\begin{equation}\label{eqn:GenFWResultPf2}
  D_t  = \|\mathcal{B}(h^t-x^t)\|_{\mathcal{B}(x^t)} \geq |\langle \nabla f(\mathcal{B}(x^t)), \mathcal{B}(h^t-x^t)\rangle|/\sqrt{\theta}.
\end{equation}
Combining~\eqref{eqn:GenFWResultPf1} and~\eqref{eqn:GenFWResultPf2}, we bound $D_t$ as
\begin{equation}
  D_t \geq \frac{G_t^a - |g(h^t) - g(x^t)|}{\sqrt{\theta}} \underset{(i)}{\geq} \frac{G_t^a - R_g}{\sqrt{\theta}} \underset{(ii)}{>} \sqrt{\theta} \underset{(iii)}{\geq} 1,
\end{equation}
where (i) uses the definition of $R_g$~\eqref{eqn:DefRh}, (ii) uses the fact that $G_t^a > \theta + R_g$ and (iii) uses~\eqref{eqn:SCProp5}.
Now, since $D_t \geq 1$, it follows that $\gamma_t < 1$, and so, substituting $\gamma_t = \frac{G_t^a}{D_t(G_t^a+D_t)}$ in~\eqref{eqn:itrIneq} and rearranging gives
\begin{equation}\label{eqn:ConstEProof1}
  F(x^{t+1}) \leq F(x^t) - \omega^{\star}\left(\frac{G_t^a}{D_t}\right).
\end{equation}
The relationship between the optimality gap at consecutive iterations is given as
\begin{equation}\label{eqn:ConstEProof2}
 \begin{split}
\Delta_t - \Delta_{t+1} &= F(x^t) - F(x^{t+1})\\ &\geq \omega^{\star}\left(\frac{G_t^a}{D_t}\right) \underset{(i)}{\geq} \omega^{\star}\left(\frac{G_t^a}{G_t^a+\theta+R_g}\right)\geq 0,
\end{split}
\end{equation}
where (i) follows from~\eqref{eqn:RelationDkGk} and the monotonicity of $\omega^{\star}(\cdot)$.
Since $G_t^a > \theta + R_g$, we have $\frac{G_t^a}{G_t^a + \theta + R_g} > \frac{1}{2}$. Now,~\cite[Proposition~2.1]{zhao2022analysis} states that $\omega^{\star}(s) \geq s/5.3$ for all $s\geq 1/2$, and so,
\begin{equation}\label{eqn:ConstEProof3}
 \Delta_t - \Delta_{t+1} \geq \frac{G_t^a}{5.3(G_t^a +\theta+R_g)}> \frac{G_t^a}{5.3\times 2G_t^a} = \frac{1}{10.6},
\end{equation}
where the second inequality follows since $G_t^a> \theta+R_g$, thus proving~\eqref{eqn:CaseIKeyEqn}.

\paragraph{Proving~\eqref{eqn:CaseIIKeyeqn}.}
First, consider the case where $\gamma_t = 1$, which indicates that $D_t(D_t+G_t^a) \leq G_t^a$, which can be rearranged as
\begin{equation}\label{eqn:C2Proof1}
  G_t^a \geq \frac{D_t^2}{1-D_t}.
\end{equation}
From~\eqref{eqn:itrIneq}, we see that
\begin{equation}\label{eqn:C2Proof2}
 \begin{split}
\Delta_t - \Delta_{t+1} &= F(x^t) - F(x^{t+1})\\ &\geq G_t^a - \omega(D_t)\\ &= G_t^a + D_t + \log (1-D_t).
\end{split}
\end{equation}
Combining~\eqref{eqn:C2Proof1},~\eqref{eqn:C2Proof2}, and the fact that $\log(1+s) \geq s - \frac{s^2}{2(1-|s|)}\quad \forall s\in (-1,1)$~\cite[Proposition~2.2]{zhao2022analysis}, we have
\begin{equation}
 \Delta_t - \Delta_{t+1} \geq G_t^a - \frac{D_t^2}{2(1-D_t)}\geq \frac{G_t^a}{2}.
\end{equation}
Furthermore,
\begin{equation}
\Delta_t - \Delta_{t+1} \geq \frac{G_t^a}{2}\underset{(i)}{\geq} \frac{(G_t^{a})^2}{2(\theta+R_g)}\underset{(ii)}{\geq} \frac{(G_t^{a})^2}{2(\theta+R_g)^2}\geq \frac{(G_t^{a})^2}{12(\theta+R_g)^2},
\end{equation}
where (i) uses the fact that $G_t^a \leq \theta + R_g$, and (ii) uses the inequality $\theta + R_g\geq 1$. This completes the proof of~\eqref{eqn:CaseIIKeyeqn} for the case $\gamma_t = 1$.
Now, if $\gamma_t < 1$, then $\gamma_t = \frac{G_t^a}{D_t(G_t^a+D_t)}$, and from~\eqref{eqn:ConstEProof1}, we have
\begin{equation}
\begin{split}
 \Delta_t - \Delta_{t+1} &= F(x^t) - F(x^{t+1})\\ &\geq \omega^{\star}\left(\frac{G_t^a}{D_t}\right)\\
&\underset{(i)}{\geq} \omega^{\star}\left(\frac{G_t^a}{G_t^a + \theta +R_g}\right)\\
&\underset{(ii)}{\geq} \frac{(G_t^{a})^2}{3(G_t^a+\theta + R_g)^2}\\
&\underset{(iii)}{\geq} \frac{(G_t^{a})^2}{12(\theta+R_g)^2},
\end{split}
\end{equation}
where (i) uses~\eqref{eqn:RelationDkGk} and the monotonicity of $\omega^{\star}(\cdot)$, (ii) uses the fact that $\frac{G_t^a}{G_t^a + \theta+R_g}\leq \frac{1}{2}$ and $\omega^{\star}(s) \geq s^2/3 \quad\forall s\in [0,1/2]$~\cite[Proposition~2.1]{zhao2022analysis}, and (iii) uses the fact that $G_t^a \leq \theta+R_g$. This completes the proof of~\eqref{eqn:CaseIIKeyeqn}.

Finally, from~\eqref{eqn:CaseIKeyEqn} and~\eqref{eqn:CaseIIKeyeqn}, we see that $(\Delta_t)_{t\geq 0}$ is a monotonically non-increasing sequence.
\end{proof}

\begin{prop}\label{prop:OutputAGFW}
If $0<\epsilon\leq \theta+R_g$ and Algorithm~\ref{algo:AGFW} stops (i.e., it reaches Step 8) at iteration $t$, we have $\Delta_t\leq G_t^a+\delta_t \leq 5\epsilon/2$.
\end{prop}

\begin{proof}
If Algorithm~\ref{algo:AGFW} stops, we have that $0\leq G_t^a\leq  \epsilon \leq \theta+R_g$ and $\delta_t\leq 3\epsilon/2$. From Lemma~\ref{lemma:relationGtGta} and~\eqref{eqn:RelationdeltaGk}, we see that $\Delta_t \leq G_t \leq G_t^a +\delta_t \leq 5\epsilon/2$.
\end{proof}

Note that, as stated in Remark~\ref{remark:Strategy2}, we can replace $\epsilon/2$ by $\epsilon/\kappa$ in the definition of adaptive strategy~\ref{Strat2}. In this regime, Algorithm~\ref{algo:AGFW} stops when $\delta_t\leq (1+1/\kappa)\epsilon$, Thus, we have that, if Algorithm~\ref{algo:AGFW} stops at iteration $t$, then $\Delta_t \leq (2+1/\kappa)\epsilon$ and $X^t$ is an $(2+1/\kappa)\epsilon$-optimal solution to the input problem. For simplicity, we assume $\kappa = 2$ in the rest of the paper.

\subsection{Convergence Analysis of Algorithm~\ref{algo:AGFW}}\label{sec:Analysis-AGFW}

Before providing the convergence analysis of Algorithm~\ref{algo:AGFW}, we introduce a few key concepts and relationships that are used in the analysis. 

\subsubsection{Partitioning the iterates of Algorithm~\ref{algo:AGFW}}
Our analysis depends on decomposing the sequence of iterates generated by Algorithm~\ref{algo:AGFW} into three subsequences depending on the value of $G_t^a$. 
This partitioning depends on a parameter $C\geq 0$, which controls which iterates fall into which subsequence, as well as the logarithmic homogeneity parameter $\theta$ and the maximum variation parameter $R_g$. Note that the definition makes sense for any pair of sequences $(x^t)_{t\geq 0}$ and $(h^t)_{t\geq 0}$ of points feasible for~\eqref{eqn:SC-gendef}, irrespective of how they are generated. 
\begin{definition}
\label{list:SeqIterates}
    Let $(x^t)_{t\geq 0}$ and $(h^t)_{t\geq 0}$ be any sequence of feasible points for~\eqref{eqn:SC-gendef} such that $G_t^a\geq 0$. Given $C\geq 0$, define three subsequences of $\mathbb{N}$ as follows:
    \begin{itemize}\label{list:SeqIterates}
    \item The \emph{$r$-subsequence} $r_1<r_2<\cdots $ consists of those $t\in \mathbb{N}$ for which $G_t^a > \theta+R_g$.
    \item The \emph{$s$-subsequence} $s_1<s_2<\cdots $ consists of those $t\in \mathbb{N}$ for which $\theta+R_g \geq G_t^a$ and $G_t^a\geq \Delta_t/(2C+1)$.
    \item The \emph{$q$-subsequence} $q_1<q_2<\cdots $ consists of those $t\in \mathbb{N}$ for which $\theta+R_g \geq G_t^a$ and $G_t^a < \Delta_t/(2C+1)$.
\end{itemize}
\end{definition}
\begin{remark}
The partitioning of the iterates as in Definition~\ref{list:SeqIterates} refines the partitioning used in~\cite{zhao2022analysis} which considers the situation where the linearized subproblems are solved exactly at each iteration. In that situation, $G_t^a = G_t \geq \Delta_t$ for all $t$ and the partitioning in Definition~\ref{list:SeqIterates} (with $C=0$) only yields the $r$- and $s$-subsequences. One key idea in our analysis is to distinguish the iterates where $G_t^a$ behaves like $G_t$ by providing an upper bound on $\Delta_t$ (the $s$-subsequence), and the iterates where $G_t^a$ does not give such an upper bound on the optimality gap $\Delta_t$ (the $q$-subsequence). For the $s$-subsequence, our analysis broadly follows the approach in~\cite{zhao2022analysis}, showing progress in both the optimality gap and the approximate duality gap. For the $q$-subsequence, we can still ensure progress in terms of the approximate duality gap. If the algorithm is run for sufficiently many iterations, either the $s$-subsequence, or the $q$-subsequence, will be long enough to ensure the stopping criterion is reached. 
\end{remark}

Given a positive integer $N$, define
\begin{align}\label{eqn:DefNrNsNt}
    N_r & = \max\{ j\;:\; r_j \leq N\}\\
    N_s & = \max\{ j\;:\; s_j \leq N\}\\
    N_q & = \max\{ j\;:\; q_j \leq N\}
\end{align}
Note that $N_r+N_s+N_q=N$, since every iterate falls into exactly one of these three sequences.

When the iterates $(x^t)_{t\geq 0}$ and $(h^t)_{t\geq 0}$ are generated by Algorithm~\ref{algo:AGFW}, 
then the optimality gap and approximate duality gap for  different subsequences satisfy certain inequalities. These allow us to analyse the progress of the algorithm  in different ways along the different subsequences. The following result summarizes the key inequalities for each subsequence that will be used in our later analysis.
\begin{lemma}
\label{lem:sub-ineq}
Let $(x^t)_{t\geq 0}$ and $(h^t)_{t\geq 0}$ be sequences of feasible points for~\eqref{eqn:SC-gendef} generated by Algorithm~\ref{algo:AGFW}. Then, for all $i\geq 1$, 
\begin{align}
    \textup{$r$-subsequence:}\quad\Delta_{r_{i}} - \Delta_{r_{i+1}} &\geq \frac{1}{10.6}\\
    \textup{$s$-subsequence:}\quad\Delta_{s_{i}} - \Delta_{s_{i+1}} &\geq \frac{(G_{s_i}^a)^2}{12(\theta+R_g)^2}
    \\
    \textup{$q$-subsequence:}\quad\quad\quad\;\;\quad G_{q_i}^a & \leq \frac{1}{2C}\delta_{q_i}\label{eq:q-ineq}
\end{align}
\end{lemma}
\begin{proof}
First, recall from Proposition~\ref{prop:DeltaConv} that the sequence $(x^t)$ generated by Algorithm~\ref{algo:AGFW} satisfies $\Delta_{t+1}\leq \Delta_t$ for all $t\geq 0$. Since $r_{i+1} \geq r_i + 1$ we have that 
\[ \Delta_{r_i} - \Delta_{r_{i+1}} \geq \Delta_{r_i} - \Delta_{r_{i}+1} \geq \frac{1}{10.6},\]
where the last inequality follows from Proposition~\ref{prop:DeltaConv}. Similarly, since $s_{i+1} \geq s_i + 1$ we have that 
\[ \Delta_{s_i} - \Delta_{s_{i+1}} \geq \Delta_{s_i} - \Delta_{s_{i}+1} \geq \frac{(G_{s_i}^a)^2}{12(\theta+R_g)^2},\]
where the last inequality follows from Proposition~\ref{prop:DeltaConv}. 

Since $\theta+R_g \geq G_{q_i}^a$, it follows from the definition of \texttt{ILMO} that condition~\ref{Cond2} ($\delta_t$-suboptimality) holds, and so  Fact~\ref{fct:gap-bnd} and Lemma~\ref{lemma:relationGtGta} tell us that $\Delta_{q_i} \leq G_{q_i}\leq G_{q_i}^a + \delta_{q_i}$. Since $\Delta_{q_i}/(2C+1) \geq G_{q_i}^a \geq 0$ (by the definition of the $q$-subsequence),  we have that 
\[ G_{q_i}^a \leq \frac{\Delta_{q_i}}{2C+1} \leq \frac{G_{q_i}}{2C+1} \leq \frac{G_{q_i}^a + \delta_{q_i}}{2C+1}.\]
Rearranging gives the inequality $G_{q_i}^a \leq \delta_{q_i}/(2C)$. 
\end{proof}

\subsubsection{Bounding the length of the $r$-subsequence}

We now provide an upper bound on the length of the $r$-subsequence in terms of the initial optimality gap $\Delta_0$.

\begin{lemma}[Bound on $N_r$]\label{lemma:boundNr}
Let $(x^t)_{t\geq 0}$ and $(h^t)_{t\geq 0}$
be sequences of feasible points for~\eqref{eqn:SC-gendef} generated by Algorithm~\ref{algo:AGFW}. For any positive integer $N$, let $N_r$ be defined as in~\eqref{eqn:DefNrNsNt}. Then
$N_r \leq \lfloor 10.6 \Delta_0\rfloor$.
\end{lemma}

\begin{proof}
From Lemma~\ref{lem:sub-ineq} we have that 
\[ \Delta_{r_{i+1}} \leq \Delta_{r_i} - \frac{1}{10.6}\quad\textup{for all $r\geq 1$}.\]
Since $r_1 \geq 0$ we have that $\Delta_{r_1}\leq \Delta_0$.  
Substituting $i=N_r-1$
and recursively applying the bound gives
\[\Delta_{r_{N_r}} \leq \Delta_{r_1} - \frac{N_r-1}{10.6} \leq \Delta_0 - \frac{N_r-1}{10.6}.\]
We also have that $ 0\leq \Delta_{r_{N_r}+1}\leq \Delta_{r_{N_r}} - \frac{1}{10.6}$, from Proposition~\ref{prop:DeltaConv}. Combining these gives
$0 \leq \Delta_0 - \frac{N_r}{10.6}$.  
It follows that $N_r \leq \lfloor 10.6 \Delta_0\rfloor$, since $N_r$ is an integer.
\end{proof}

\subsubsection{Controlling the optimality gap and approximate duality gap along the $s$-subsequence}

In this section, we will show that in any 
sufficiently large interval of the $s$-subsequence, there must be an iterate at which the approximate duality gap is small. To do this, 
we use the following result, a slight modification of~\cite[Proposition~2.4]{zhao2022analysis}.  While~\cite[Proposition~2.4]{zhao2022analysis} bounds the minimum over the first $j+1$ elements in the sequence $(g_j)$, i.e., $\min\{g_0,\dotsc,g_j\}$, Proposition~\ref{prop:ProofGkConv} provides a bound on the minimum over any subsequence of consecutive elements $g_l,\dotsc,g_j$ with $j>l\geq 0$.
\begin{prop}\label{prop:ProofGkConv}
Let $\beta>0$. Consider two nonnegative sequences $(d_j)_{j\geq 0}$ and $(g_j)_{j\geq 0}$ that satisfy
\begin{itemize}
\item $d_{j+1}\leq d_j - g_j^2/\beta$ for all $j\geq 0$ and
\item $g_j \geq d_j$ for all $j\geq 0$.
\end{itemize}
If $j\geq 1$, then $d_j < \beta/j$ and 
\begin{equation}
\min \{g_l,\dotsc, g_j\} < \begin{cases}
    \frac{2\beta}{j} & \textup{if $\lfloor \frac{j+1}{2}\rfloor>l\geq 0$}\\
    \frac{\beta}{\sqrt{(l+1)(j-l)}} & \textup{if $\lfloor \frac{j+1}{2}\rfloor\leq l < j$}
\end{cases}
\end{equation}
If, in addition, $d_0 \leq \beta$, then $d_j\leq \beta/(j+1)$ and 
\begin{equation}
\min \{g_l,\dotsc, g_j\} \leq \begin{cases}
    \frac{2\beta}{j+1} & \textup{if $\lfloor \frac{j}{2}+1\rfloor>l\geq 0$}\\
    \frac{\beta}{\sqrt{(l+2)(j-l)}} & \textup{if $\lfloor \frac{j}{2}+1\rfloor\leq l < j$}\end{cases}
    \end{equation}
\end{prop}
\begin{proof}
First, suppose that $d_j = 0$ for some $j$. In this case $g_k=0$ for all $k\geq j$ and $d_k = 0$ for all $k\geq j$, so the bounds hold trivially. As such, assume that $j>l\geq 0$ is chosen so that $d_{j}>0$, and hence $d_k>0$ for $0\leq k\leq j$. 

The first recurrence implies that  $d_{j}\leq d_{j-1}$ for all $j\geq 1$.  
Dividing both sides of the first recurrence by the positive quantity $d_{j}d_{j-1}$ gives 
    \[ d_{j}^{-1} \geq d_{j-1}^{-1} + \beta^{-1}\frac{g_{j-1}}{d_{j-1}}\frac{g_{j-1}}{d_{j}} \geq d_{j-1}^{-1} + \beta^{-1} \frac{g_{j-1}}{d_{j}} \geq d_{j-1}^{-1} + \beta^{-1}\]
    where the inequalities follow from the assumption that $g_{j-1} \geq d_{j-1}$ and the fact that $d_{j}\leq d_{j-1}$ for all $j\geq 0$.

It follows that $d_{j}^{-1} \geq d_l^{-1} + (j-l)\beta^{-1}$ and so $d_j\leq\frac{\beta}{j-l+\beta/d_l}$. Since this inequality is independent of $l$, we can minimize over $l$ to obtain the best bound. By setting $l=0$ and using $d_0> 0$, we obtain $d_j < \beta/j$ for all $j\geq 1$. Furthermore, if $d_0\leq \beta$, we have that $d_j \leq \beta/(j+\beta/d_0) \leq \beta/(j+1)$.

Now we establish the bound on $\min\{g_l,\ldots,g_j\}$. Let $1\leq l+1\leq i\leq j$.
Since the minimum is bounded above by the average, we have that 
\[ \min\{g_l^2,\ldots,g_j^2\} \leq \min\{g_{i}^2,\dots,g_j^2\} \leq  \frac{1}{j-i+1}\sum_{k=i}^{j}g_k^2 \leq \frac{\beta}{j-i+1}\sum_{k=i}^{j}(d_{k}-d_{k+1})= \frac{\beta}{j-i+1}(d_{i} - d_{j+1}).\]
Using the bound $d_{i} < \beta/i$ (valid since $i\geq 1$), we have that 
\begin{equation}
\label{eq:gb}\min\{g_l^2,\ldots,g_j^2\} < \frac{\beta}{j-i+1}\cdot\frac{\beta}{i}.\end{equation}
We are free to choose $l+1\leq i \leq j$ to minimize the bound on the right hand side of~\eqref{eq:gb}. Let $i^* = \lfloor \frac{j+1}{2}\rfloor$. A minimum occurs at $i^*$ if $i^*\geq l+1$, otherwise it occurs at $i=l+1$. In the first case, we obtain
\[ \min\{g_l^2,\ldots,g_j^2\} \leq \frac{4\beta^2}{j^2}.\]
In the second case we obtain
\[ \min\{g_l^2,\ldots,g_j^2\} \leq \frac{\beta^2}{(l+1)(j-l)}.\]
Taking the square root of both sides gives the desired bound.

If, in addition $d_0\leq \beta$, we can use the bound $d_i \leq \beta/(i+1)$ which gives
\begin{equation}
\label{eq:gb2}\min\{g_l^2,\ldots,g_j^2\} \leq \frac{\beta}{j-i+1}\cdot \frac{\beta}{i+1}.
\end{equation}
Again we want to minimize the right hand size of~\eqref{eq:gb2} over $l+1\leq i\leq j$. Let $i^* = \lfloor j/2+1\rfloor$. A minimum occurs at $i=i^*$ if $i^*\geq l+1$, otherwise a minimum occurs at $i=l+1$. Substituting these into~\eqref{eq:gb} and taking the square roots of both sides completes the argument.
\end{proof}

The following lemma shows how the optimality gap and the approximate duality gap reduce for iterates in the $s$-subsequence. The implication of this result is that if the $s$-subsequence is long enough, both of these quantities will eventually become small.

\begin{lemma}
\label{lemma:boundNs}
Let $(x^t)_{t\geq 0}$ and $(h^t)_{t\geq 0}$
be sequences of feasible points for~\eqref{eqn:SC-gendef} generated by Algorithm~\ref{algo:AGFW}. 
\begin{itemize}
    \item If $j \geq 1$ then $\Delta_{s_{j+1}} \leq \frac{12(2C+1)^2(\theta+R_g)^2}{j+1}$.
    \item If $0\leq l < \lfloor j/2+1\rfloor$ then $\displaystyle{\min_{k\in \{l+1,\dotsc, j+1\}}}G_{s_k}^a \leq \frac{24(2C+1)(\theta+R_g)^2}{j+1}$.
    \item If $l\geq  \lfloor j/2+1\rfloor$ then $\displaystyle{\min_{k\in \{l+1,\dotsc, j+1\}}}G_{s_k}^a \leq \frac{12(2C+1)(\theta+R_g)^2}{\sqrt{(l+2)(j-l)}}$.
    \end{itemize}

 \end{lemma}
 \begin{proof}
Consider the sequences $d_{j-1} = \Delta_{s_{j}}/(2C+1)$ and $g_{j-1} = G_{s_{j}}^a$ for $j=1,2,\ldots$. 
Note that $d_j \geq 0$ and $g_j \geq 0$ for all $j$. 
From Lemma~\ref{lem:sub-ineq}
we have that
 \begin{equation}
     d_j 
     \leq d_{j-1} - \frac{g_{j-1}^2}{12(2C+1)(\theta+R_g)^2}\quad\textup{for all $j\geq 1$}.
     \end{equation}
Combining this with the inequality $g_{j-1} = G_{s_j}^a \geq \Delta_{s_j}/(2C+1) = d_{j-1}$ for all $j=1,2,\ldots$ (from the definition of the $s$-subsequence), 
 we have that 
 \begin{equation*}
     d_{j+1} \leq d_j - \frac{g_j^2}{\beta}\quad\textup{and}\quad     g_j  \geq d_j\quad\textup{for all $j\geq 0$},
 \end{equation*}
 where $\beta = 12(2C+1)(\theta+R_g)^2$.
 Furthermore, $d_0 \leq g_0 = G_{s_1}^a \leq \theta + R_g \leq \beta$ (since $12(2C+1)(\theta+R_g)^2 \geq 1$). 
 The result then follows from Proposition~\ref{prop:ProofGkConv}.
%
\end{proof}

\subsection{Analysis of Algorithm~\ref{algo:AGFW}}\label{sec:ConvAGFW}
In this section, we provide a bound on the number of iterations of Algorithm~\ref{algo:AGFW} required to achieve the stopping criterion and return an $\epsilon$-optimal solution (i) when $\delta_t$ is chosen according to scheduled strategy~\ref{Strat1} (see Lemma~\ref{lemma:ConvDet-OL}), and (ii) when $\delta_t$ is chosen according to adaptive strategy~\ref{Strat2} (see Lemma~\ref{lemma:ConvDet-CL}). 

\begin{lemma}\label{lemma:ConvDet-OL}
    Let $0<\epsilon\leq\theta+R_g$ and let $\delta_t$ be a non-negative sequence that converges to $0$. Let $K_q$ be the smallest positive integer such that $\delta_t \leq \epsilon/2$ for all $t\geq K_q$. Let 
    \[ K_r = \lfloor 10.6\Delta_0\rfloor\quad\textup{and}\quad K_s = \left\lceil \frac{48(\theta+R_g)^2}{\epsilon}\right\rceil.\]
    Then, for Algorithm~\ref{algo:AGFW} (with~\eqref{eqn:SC-gendef} as input), there exists some $K \leq K_r+2K_s+2K_q$ such that $\Delta_K \leq \epsilon$,
    $G_K^a\leq \epsilon$ and $\delta_K \leq \epsilon$. 
\end{lemma}

\begin{proof}
Let $N = K_r+2K_s+2K_q$ and let the first $N$ iterations of Algorithm~\ref{algo:AGFW} be partitioned into the $r$-, $s$-, and $q$-subsequences according to Definition~\ref{list:SeqIterates} with $C = 1/2$. Then $N = N_r+N_s+N_q$.

From Lemma~\ref{lemma:boundNr}, we know that $N_r \leq K_r$, and so it follows that $N_s +N_q \geq (2K_s+K_q)+K_q$. We now consider two cases: (i) $N_s \geq  2K_s+K_q+1$, and (ii) $N_s \leq 2K_s+K_q$. In the first case, the $s$-subseqence will be long enough to establish convergence. In the second case, the $q$-subsequence will be long enough to ensure convergence. 

\textbf{Case 1: $N_s \geq   2K_s+K_q+1$.} 
In this case we will show that it suffices to choose $K = s_{k^*}$ where 
\[ k^* = \argmin_{i\in \{K_s+K_q+1,\ldots,2K_s+K_q+1\}} G_{s_i}^a.\] Since $N_s\geq 2K_s+K_q+1$ it follows that $K= s_{k^*} \leq s_{2K_s+K_q+1} \leq s_{N_s} \leq N$, as required. We now establish bounds on $\delta_K, G_K^a$, and $\Delta_K$. 
\begin{itemize}
\item Since $k^*\geq K_q$, it follows that $s_{k^*} \geq s_{K_q} \geq K_q$, and so that $\delta_{s_k^*}\leq \epsilon/2$ (by the definition of $K_q$). 
\item  By the definition of $k^*$, from Lemma~\ref{lemma:boundNs} (with $l=K_s+K_q$ and $j=2K_s+K_q$) we have that $l\geq \lfloor j/2+1\rfloor$ and so 
\[ G_{s_{k^*}}^a \leq \frac{24(\theta+R_g)^2}{\sqrt{(K_s+K_q+2)K_s}}\leq \epsilon/2.\]
\item From Proposition~\ref{prop:OutputAGFW} we have that $\Delta_{s_{k^*}} \leq G_{s_{k^*}}^a + \delta_{s_{k^*}}\leq \epsilon$.
\end{itemize}




\textbf{Case 2: $N_s \leq  2K_s+K_q$.} In this case we will show that it suffices to  choose $K = q_{N_q}\leq N$. 

Since $N_s \leq  2K_s+K_q$, we have that $N_q \geq 2K_s+2K_q - N_s\geq K_q$.
Consider the sequence $q_k$ for $k=1,2,\ldots,N_q$. Since $C = 1/2$, from Lemma~\ref{lem:sub-ineq} we have that 
    $G_{q_k}^a \leq \delta_{q_k}$ for each $k$.
    
    Since $K = q_{N_q} \geq N_q \geq K_q$, we have $\delta_{K} = \delta_{q_{N_q}}\leq \epsilon/2$ (from the definition of $K_q$).
    We also have $G_{K}^a = G_{q_{N_q}}^a \leq \delta_{q_{N_q}} \leq \epsilon/2$. 
    Finally, from Proposition~\ref{prop:OutputAGFW} we have that $\Delta_{K} \leq G_{K}^a +\delta_{K} \leq \epsilon$, completing the argument. 
\end{proof}

The following result (Lemma~\ref{lemma:ConvDet-CL}) provides an upper bound on the number of iterations of Algorithm~\ref{algo:AGFW} until it stops when the value of $\delta_t$ is set according to adaptive strategy~\ref{Strat2}. The main difference between the proofs of Lemma~\ref{lemma:ConvDet-OL} and Lemma~\ref{lemma:ConvDet-CL} is in the analysis of the $q$-subsequence. In Lemma~\ref{lemma:ConvDet-CL}, since $\delta_t$ is chosen in terms of the smallest value of the approximate duality gap we have encountered so far, we can ensure that the approximate duality gap reduces geometrically along the $q$-subsequence. 

\begin{lemma}\label{lemma:ConvDet-CL}
    Let $0<\epsilon\leq \theta+R_g$,  and let  $\delta_t$ be a non-negative sequence such that $\delta_t \leq \epsilon/2 + \min_{\tau<t}G_\tau^a$ for all $t$. Let 
    \[ K_r = \left\lfloor 10.6 \Delta_0\right\rfloor\quad\textup{and}\quad
K_s = \left\lceil \frac{72(\theta+R_g)^2}{\epsilon}\right\rceil\quad\textup{and}\quad K_q = 2+\left\lceil\log_2\left(\frac{\theta+R_g}{\epsilon}\right)\right\rceil.\]
Then, for Algorithm~\ref{algo:AGFW} (with~\eqref{eqn:SC-gendef} as input), there exists some $K \leq K_r+2K_s+K_q$ such that $\Delta_K \leq 5\epsilon/2$,
    $G_K^a\leq \epsilon$ and $\delta_K \leq 3\epsilon/2$. 
\end{lemma}


\begin{proof}
Let $N = K_r+2K_s+2K_q$ and let the first $N$ iterations of Algorithm~\ref{algo:AGFW} be partitioned into the $r$-, $s$-, and $q$-subsequences according to Definition~\ref{list:SeqIterates} with $C = 1$. Then $N = N_r+N_s+N_q$.
We know that $N_r \leq K_r$ (from Lemma~\ref{lemma:boundNr}). It follows that $N_s +N_q \geq 2K_s+2K_q$.

We now consider two cases: (i) $N_q \leq K_q$, (ii) $N_q \geq K_q+1$.

\textbf{Case 1: $N_q \leq K_q$.} 
In this case we will show that it suffices to choose $K = s_{k^*}$ where 
\[ k^* = \argmin_{i\in \{K_s+1,\ldots,2K_s\}} G_{s_i}^a.\]
Since $N_q \leq K_q$, we have $N_s \geq 2K_s+K_q-N_q \geq 2K_s \geq k^*$. It follows that $K = s_{k^*} \leq s_{N_s} \leq N$, as required. We now establish bounds on $\delta_K, G_K^a$ and $\Delta_K$.
\begin{itemize}
\item Since $k^*\geq K_s+1$ we have that
\[\delta_{s_{k^*}} \leq \frac{\epsilon}{2}+\min_{\tau<s_{k^*}}G_{\tau}^a \leq \frac{\epsilon}{2}+\min_{1\leq j\leq k^*-1}G_{s_j}^a\leq  \frac{\epsilon}{2}+\min_{1\leq j\leq K_s}G_{s_j}^a \underset{(i)}{\leq} \frac{\epsilon}{2}+\frac{72(\theta+R_g)^2}{K_s}\underset{(ii)}{\leq} \frac{3\epsilon}{2},\]
where (i) follows from Lemma~\ref{lemma:boundNs} with $l=0$ and $j=K_s-1$ and (ii) follows from the definition of $K_s$. 
\item From Lemma~\ref{lemma:boundNs} with $l=K_s$ and $j = 2K_s-1$, noting that $l=\lfloor j/2+1\rfloor$, we have that 
\[ G_{s_{k^*}}^a \leq \frac{36(\theta+R_g)^2}{\sqrt{(K_s+2)(K_s-1)}}\leq \frac{36(\theta+R_g)^2}{K_s} \leq \frac{\epsilon}{2}.\]
\item 
From Proposition~\ref{prop:OutputAGFW} we have that $\Delta_{s_{k^*}} \leq G_{s_{k^*}}^a + \delta_{s_{k^*}}\leq 5\epsilon/2$.
\end{itemize} 

\textbf{Case 2: $N_q \geq K_q+1$.} In this case, we will show that it suffices to choose $K=q_{N_q}\leq N$.

Consider the sequence $q_k$ for $k=1,2,\ldots,N_q$. Since $C=1$, from Lemma~\ref{lem:sub-ineq} we have that $G_{q_k}^a \leq \delta_{q_k}/2$
    for $k=1,2,\ldots,N_q$. From our choice of $\delta_t$, we know that for $k\geq 2$, 
    \[ \delta_{q_k} \leq \frac{\epsilon}{2}+\min_{\tau<q_k}G_\tau^a\leq \frac{\epsilon}{2}+\min_{\ell<k}G_{q_\ell}^a \leq 
    \frac{\epsilon}{2}+G_{q_{k-1}}^a.\]
    Since $G_{q_1}^a\leq \theta + R_g$ and, for each $k\geq 2$ we have
    $G_{q_k}^a \leq \epsilon/4+(1/2)G_{q_{k-1}}^a$, it follows that \[ G^a_{q_k}\leq \frac{\epsilon}{4}\left(\sum_{i=0}^{k-2}\frac{1}{2^i}\right)+ 2^{-k+1}(\theta+R_g) \leq \frac{\epsilon}{2}+2^{-k+1}(\theta+R_g)\quad\textup{for $k\geq 2$},\]
    where the second inequality holds because $\sum_{i=0}^{\ell}2^{-i} \leq 2$ for all $\ell \geq 0$.    
Recall that $K=q_{N_q}\leq N$ and that $K_q+1\leq N_q$ by assumption. Then 
\[ G_{K}^a = G_{q_{N_q}}^a \leq \frac{\epsilon}{2} + 2^{-N_q+1}(\theta+R_g)\leq \frac{\epsilon}{2}+ 2^{-K_q}(\theta+R_g)\leq \frac{\epsilon}{2}+\frac{\epsilon}{4}= \frac{3\epsilon}{4}\]
by the definition of $K_q$. 
Moreover 
\[ \delta_{K} = \delta_{q_{N_q}} \leq \frac{\epsilon}{2}+G_{q_{(N_q-1)}}^a \leq \frac{\epsilon}{2}+ \frac{\epsilon}{2}+  2^{-N_q+2}(\theta+R_g) \leq \epsilon+2^{-K_q+1}(\theta+R_g)\leq 3\epsilon/2.\] 
Finally, from Proposition~\ref{prop:OutputAGFW} we have that 
$\Delta_{K} \leq G_K^a + \delta_K \leq 
9\epsilon/4\leq 5\epsilon/2$, as required. 
\end{proof}

\section{Implementing \texttt{ILMO}}
\label{sec:AddMultiErr}
 In this section we discuss how to implement the oracle \texttt{ILMO} given that we have access to a method to approximately solve~\eqref{prob:LinearMinimizationProblem} with either prescribed additive error or (under additional assumptions) prescribed relative error. 
 
 If we have a method to solve~\eqref{prob:LinearMinimizationProblem} to a prescribed additive error, then we can use it to simply guarantee that condition~\ref{Cond2} ($\delta_t$-suboptimality) is satisfied at each iteration. However, this does not automatically guarantee that the associated near-minimizer for~\eqref{prob:LinearMinimizationProblem} satisfies the additional condition $G_t^a \geq 0$. In Section~\ref{sec:Gta0}, we briefly discuss how to ensure that this condition is also satisfied.

The oracle \texttt{ILMO}, with its two possible exit conditions (~\ref{Cond1} (large gap) and~\ref{Cond2} ($\delta_t$-suboptimality)), allows more implementation 
flexibility than just solving~\eqref{prob:LinearMinimizationProblem} to a prescribed additive error. In Section~\ref{sec:MultiErrorOracle}, we restrict to problems for which $g$ is the indicator function of a compact convex set. In this case, we will show that if we only have a method to solve~\eqref{prob:LinearMinimizationProblem} to prescribed \emph{relative error}, it can still be used to implement \texttt{ILMO}. This is less straightforward since, in general, the optimal value of~\eqref{prob:LinearMinimizationProblem} is not uniformly bounded. Moreover, this is exactly the scenario of interest in the motivating example discussed in Section~\ref{sec:Intro-Issues}. 

\subsection{Ensuring the output of \texttt{ILMO} satisfies $G_t^a\geq0$}
\label{sec:Gta0}
The definition of \texttt{ILMO} requires that the point $h^t$ produced by \texttt{ILMO} is such that  $G_t^a(x^t,h^t)$ is nonnegative. 
This is clearly satisfied whenever \texttt{IMLO} returns a point $h^t$ such that $G_t^a \geq \theta+R_g$. However, 
when the prescribed additive error $\delta_t$ for~\eqref{prob:LinearMinimizationProblem} is small, it could be the case that condition~\ref{Cond2} ($\delta_t$-suboptimality) is satisfied but $G_t^a<0$. In this case, it turns out that we can 
just return $h^t = x^t$ as a valid output of \texttt{ILMO}. 
\begin{lemma}
\label{lem:Gta0}
Suppose that $x^t$ and $h^t$ are feasible for~\eqref{eqn:SC-gendef} and satisfy $F^{\textup{lin}}_{x^t}(h^t) - F^{\textup{lin}}_{x^t}(h^\star)\leq \delta_t$, (i.e., condition~\ref{Cond2} ($\delta_t$-suboptimality)). If $G_t^a(x^t,h^t)<0$ then $F^{\textup{lin}}_{x^t}(x^t) - F^{\textup{lin}}_{x^t}(h^\star) \leq \delta_t$ and $G_t^a(x^t,x^t) = 0$. 
\end{lemma}
\begin{proof}
    Recall from the definition of $F^{\textup{lin}}_{x^t}$ in~\eqref{prob:LinearMinimizationProblem} and the definition of $G_t^a(x^t,h^t)$ in~\eqref{eqn:DefGkApprox} that 
    \[ G_t^a(x^t,h^t) = F^{\textup{lin}}_{x^t}(x^t) - F^{\textup{lin}}_{x^t}(h^t).\]
    If $G_t^a<0$ then $F^{\textup{lin}}_{x^t}(x^t) < F^{\textup{lin}}_{x^t}(h^t)$
    and so $x^t$ is a feasible point for~\eqref{prob:LinearMinimizationProblem} with smaller cost than $h^t$. Since $h^t$ satisfies condition~\ref{Cond2} ($\delta_t$-suboptimality), the same holds for $x^t$, since
    \[ F^{\textup{lin}}_{x^t}(x^t) - F^{\textup{lin}}_{x^t}(h^\star) \leq F^{\textup{lin}}_{x^t}(h^t) - F^{\textup{lin}}_{x^t}(h^\star) \leq \delta_t.\]
    In addition $G_t^a(x^t,x^t)=0$ (from the definition of $G_t^a$), completing the argument.
\end{proof}
We note that if \texttt{ILMO} returns $(h^t,G_t^a) = (x^t,0)$ then either Algorithm~\ref{algo:AGFW} terminates at line 7 or we will have $D_t = 0$ and $\gamma_t = 0$, so that $x^{t+1} = x^t$. 



\subsection{\texttt{ILMO} using an approximate minimizer of $F^{\textup{lin}}_{x^t}$ with prescribed relative error}
\label{sec:MultiErrorOracle}
In this section, we assume that we have access to a method that can compute an approximate minimizer of $F^{\textup{lin}}_{x^t}$ with prescribed relative error. 
We further assume that $g(\cdot)$ is an indicator function, so that $R_g = 0$. This assumption ensures that the optimal value of~\eqref{prob:LinearMinimizationProblem} is related to $\theta$, the parameter of logarithmic homogeneity of $f$.  
\begin{lemma}
\label{lem:ind-opt-bnd}
    Consider an instance of~\eqref{eqn:SC-gendef} with $R_g = 0$. Let $x^t$ and $h^t$ be feasible for~\eqref{eqn:SC-gendef}. 
    Then 
    \begin{align*}
   G_t^a(x^t,h^t) & = -\theta - F^{\textup{lin}}_{x^t}(h^t)\quad\textup{and}\\
   G_t(x^t) & = -\theta - F^{\textup{lin}}_{x^t}(h^\star).
   \end{align*}
   It follows that $F^{\textup{lin}}_{x^t}(h^\star) \leq -\theta$. 
\end{lemma}
\begin{proof}
Since $x^t$ is feasible for~\eqref{eqn:SC-gendef} and $R_g=0$, we have that $g(x^t)=0$. We then have $F^{\textup{lin}}_{x^t}(x^t) = \langle \nabla f(\mathcal{B}(x^t)),\mathcal{B}(x^t)\rangle = -\theta$ since $f$ is $\theta$-logarithmically-homogeneous (see Proposition~\ref{prop:SCLHProperties}). Using the fact that $G_t^a(x^t,h^t) = F^{\textup{lin}}_{x^t}(x^t) - F^{\textup{lin}}_{x^t}(h^t)$, we obtain the stated expression for $G_t^a(x^t,h^t)$. The expression for $G_t(x^t)$ follows by substituting $h^\star$ for $h^t$. The observation that $F^{\textup{lin}}_{x^t}(h^\star) \leq -\theta$ follows since $G_t \geq 0$ (Fact~\ref{fct:gap-bnd}). 
\end{proof}


The following result shows that a method to solve~\eqref{prob:LinearMinimizationProblem} with a multiplicative error guarantee is enough to implement \texttt{ILMO} in Algorithm~\ref{algo:AGFW}. Crucially, the prescribed multiplicative error required is uniformly bounded, independent of the current iterate $x^t$.

\begin{lemma}\label{lemma:LMODef-RelErr}
Consider an instance of~\eqref{eqn:SC-gendef} with $R_g = 0$. Let $\delta_t\geq 0$, let $c>2$ and let $\tau_t := \frac{\textup{min}\{\delta_t,(c-2)\theta\}}{c\theta}<1$. If $h^t$ is feasible for~\eqref{prob:LinearMinimizationProblem} and
\begin{equation}\label{eqn:approxFWLMO-RelErr}
F^{\textup{lin}}_{x^t}(h^\star) \leq F^{\textup{lin}}_{x^t}(h^t) \leq (1-\tau_t)F^{\textup{lin}}_{x^t}(h^\star),
\end{equation}
then at least one of the two conditions \ref{Cond1} (large gap) and~\ref{Cond2} ($\delta_t$-suboptimality) is satisfied.
\end{lemma}

\begin{proof}
We prove the result by considering two cases for the value of $OPT_t:= F^{\textup{lin}}_{x^t}(h^\star)$.
From Lemma~\ref{lem:ind-opt-bnd} we know that $-OPT_t\geq \theta$. The argument is based on two cases, depending on whether $-OPT_t$ is large or small. 
\paragraph{Case I: $-OPT_t \leq c\theta$.}
In this case, $\tau_t \leq \frac{\textup{min}\{\delta_t,(c-2)\theta\}}{c\theta} \leq \frac{\delta_t}{(-OPT_t)}$. Thus, $\tau_t(-OPT_t)\leq \delta_t$. Substituting this inequality into~\eqref{eqn:approxFWLMO-RelErr}, we see that
\begin{equation}
F^{\textup{lin}}_{x^t}(h^t) \leq
OPT_t + \delta_t,
\end{equation}
thus, satisfying the condition~\ref{Cond2} ($\delta_t$-suboptimality).

\paragraph{Case II: $-OPT_t > c\theta.$}
From Lemma~\ref{lem:ind-opt-bnd} and~\eqref{eqn:approxFWLMO-RelErr}, we have
\begin{equation}
\begin{split}
G_t^a &= -F^{\textup{lin}}_{x^t}(h^t) - \theta\\
&\geq (1-\tau_t)(-OPT_t) - \theta \\
&\underset{(i)}{>} (1-\tau_t)c\theta-\theta\\
&\underset{(ii)}{\geq} \left(1-\frac{\textup{min}\{\delta_t,(c-2)\theta\}}{c\theta}\right)c\theta - \theta\\
&= (c-1)\theta -\textup{min}\{\delta_t,(c-2)\theta\}\\
&\geq (c-1)\theta - (c-2)\theta\\
&\geq \theta,
\end{split}
\end{equation}
where (i) follows since $-OPT_t>c\theta$ and (ii) follows from the definition of $\tau_t$. Since $R_g = 0$, we have that $G_t^a \geq \theta + R_g$, showing that condition~\ref{Cond1} ($\delta_t$-suboptimality) is satisfied.
\end{proof}


\section{Probabilistic \texttt{LMO}}\label{sec:ProbLMO}
Algorithm~\ref{algo:AGFW} uses an oracle \texttt{ILMO} that generates a pair $(h^t,G^t_a)$ such that the output satisfies at least one of the two conditions~\ref{Cond1} (large gap) and~\ref{Cond2} ($\delta_t$-suboptimality). In this section, we consider a setting where we only have access to a probabilistic oracle \texttt{PLMO}, that fails with probability at most $p$ to generate the output as required by the algorithm (see Definition~\ref{def:PLMO}). 
We first provide an algorithm that uses \texttt{PLMO} to generate an $\epsilon$-optimal solution to~\eqref{eqn:SC-gendef} with any desired probability of success.
We then analyze the convergence of the algorithm for both strategies~\ref{Strat1} (scheduled) and~\ref{Strat2} (adaptive) of choosing the sequence $\delta_t$.

\paragraph{Framework of the algorithm.} Algorithm~\ref{algo:AGFW-P} provides the framework of the algorithm that uses \texttt{PLMO}. There are two main differences between Algorithm~\ref{algo:AGFW} and Algorithm~\ref{algo:AGFW-P}. The first is that we use the probablistic oracle \texttt{PLMO} (see Definition~\ref{def:PLMO}) in Step 6 instead of \texttt{ILMO}, i.e., with probability at most $p$, the pair $(h^t,G_t^a)$ generated in the Step 6  of Algorithm~\ref{algo:AGFW-P} satisfies neither condition~\ref{Cond1} (large gap) nor condition~\ref{Cond2} ($\delta_t$-suboptimality), but does satisfy $G_t^a\geq 0$. The second difference is in the stopping criteria of the algorithm. While Algorithm~\ref{algo:AGFW} stops when $G_t^a \leq \epsilon$ and $\delta_t\leq 3\epsilon/2$, Algorithm~\ref{algo:AGFW-P} stops when the bounds $G_t^a \leq \epsilon$ and $\delta_t\leq 3\epsilon/2$ are satisfied $l$ times. We will see (in Proposition~\ref{prop:OutputAGFW-P}) that the parameter $l$ controls the probability that the optimality gap is small when the algorithm stops.

\begin{algorithm}[tbh]
\caption{Approximate Generalized Frank-Wolfe with \texttt{PLMO}}
\label{algo:AGFW-P}
\textbf{Input}: Problem~\eqref{eqn:SC-gendef}, where $f(\cdot)$ is a 2-self-concordant, $\theta$-logarithmically-homogeneous barrier function, $\epsilon \leq \theta+R_g$: suboptimality parameter, scheduled strategy~\ref{Strat1} or adaptive strategy~\ref{Strat2} for setting the value of $\delta_t$\\
\textbf{Output}: $\widehat{x}_{\epsilon}$: $\widehat{x}_{\epsilon}$ an $\epsilon$-optimal solution to~\eqref{eqn:SC-gendef} with probability at least $1-p^l$
\begin{algorithmic}[1] 
\STATE Initialize $x^0 \in \mathbb{R}^n$.
\STATE Set $t\leftarrow 0$, $\textup{count} = 0$.
\WHILE{$\textup{count} < l$}
\STATE Compute $\mathcal{B}^{\star}(\nabla f(\mathcal{B}(x^t)))$.
\STATE Set $\delta_t \geq 0$ according to the input strategy.
\STATE $(h^t,G_t^a) = \texttt{PLMO} (\mathcal{B}^{\star}(\nabla f(\mathcal{B}(x^t))),\delta_t, x^t)$.
\IF {$G_t^a \leq \epsilon$ and $\delta_t \leq 3\epsilon/2$} \STATE $\textup{count} = \textup{count} + 1$.\ENDIF
\STATE Compute $D_t = \|\mathcal{B}(h^t-x^t)\|_{\mathcal{B}(x^t)}$.
\STATE Set $\gamma_t = \min \left\{ \frac{G_t^a}{D_t(D_t+G_t^a)},1\right\}$.
\STATE Update $x^{t+1} = (1-\gamma_t)x^t + \gamma_t h^t$.
\STATE $t = t+1$.
\ENDWHILE
\STATE \textbf{return $(x_t,\delta_t)$}
\end{algorithmic}
\end{algorithm}

It is sometimes fruitful to think of the iterates of Algorithm~\ref{algo:AGFW-P} as being iterates of Algorithm~\ref{algo:AGFW} with respect to an alternative (random) sequence $(\hat{\delta}_t)_{t\geq 0}$. To see how this is the case, let  $(\delta_t)_{t\geq 0}$ be a non-negative sequence and let $(b_t)_{t\geq 0}$ be a sequence of independent random variables where $b_t =1$ with probability $1-p_t$ and $b_t = \infty$ with probability $p_t$ and $p_t\leq p$ for all $t$. Then we can think of the iterates of Algorithm~\ref{algo:AGFW-P} 
with respect to $(\delta_t)_{t\geq 0}$ as the iterates of 
Algorithm~\ref{algo:AGFW} with respect to the random sequence $(\hat{\delta}_t)_{t\geq 0}$ where $\hat{\delta}_t = b_t\delta_t$ for all $t$. The multiplier $b_t$ allows \texttt{ILMO} to model the behaviour of \texttt{PLMO}. Indeed when the multiplier $b_t = \infty$, condition~\ref{Cond2} ($\delta_t$-suboptimality) holds trivially with respect to $\hat{\delta}_t = b_t\delta_t = \infty$, effectively enforcing no additional constraint on the output of the oracle, apart from $G_t^a\geq 0$. 

This viewpoint allows us to leverage existing results for the iterates of Algorithm~\ref{algo:AGFW} in the context of Algorithm~\ref{algo:AGFW-P}. In particular, any results that are independent of the sequence $(\delta_t)_{t \geq 0}$, such as Lemmas~\ref{lemma:boundNr} and~\ref{lemma:boundNs}, continue to hold for Algorithm~\ref{algo:AGFW-P}.
Other results that depend on $\delta_t$ (such as the bound~\eqref{eq:q-ineq} in Lemma~\ref{lem:sub-ineq}) remain valid as long as we replace $\delta_t$ with the random variable $\hat{\delta}_t = b_t\delta_t$. In light of this observation, the following result summarizes the key inequalities that hold for the iterates of Algorithm~\ref{algo:AGFW-P}.
\begin{lemma}
    \label{lem:key-ineq-p}
    Let $(x^t)_{t\geq 0}$ and $(h^t)_{t\geq 0}$ be sequences of feasible points for~\eqref{eqn:SC-gendef} generated by Algorithm~\ref{algo:AGFW-P}. Partition the iterates into the $q$-, $r$-, and $s$- subsequences with respect to some $C>0$ as in Definition~\ref{list:SeqIterates}. Then
    \begin{itemize}
        \item $N_r \leq \lfloor 10.6 \Delta_0\rfloor$
        \item If $0\leq l < \lfloor j/2+1\rfloor$ then $\displaystyle{\min_{k\in \{l+1,\ldots,j+1\}} G_{s_k}^a \leq \frac{24(2C+1)(\theta+R_g)^2}{j+1}}$.
        \item If $l \geq \lfloor j/2+1\rfloor$ then $\displaystyle{\min_{k\in \{l+1,\ldots,j+1\}} G_{s_k}^a \leq \frac{12(2C+1)(\theta+R_g)^2}{\sqrt{(l+2)(j-l)}}}$.
        \item The random variables $\frac{1}{2C}\delta_{q_k} - G_{q_k}^a$ (indexed by $k\geq 1$) are independent, and are each non-negative with probability at least $1-p$.
    \end{itemize}
\end{lemma}

The following result shows that the parameter $l$ can be used to control the probability that Algorithm~\ref{algo:AGFW-P} successfully returns a near-optimal point.
\begin{prop}\label{prop:OutputAGFW-P}
If $0\leq \epsilon \leq \theta+R_g$ and Algorithm~\ref{algo:AGFW-P} (with parameter $l$) stops after $T_l$ iterations, we have $\Delta_{T_l}\leq G_{T_l}^a+\delta_{T_l} \leq 5\epsilon/2$ with probability at least $1-p^l$.
\end{prop}
\begin{proof}
Let $(T_i)_{i\geq 1}$ be the iterations of Algorithm~\ref{algo:AGFW-P} at which $G_{T_i}^a \leq \epsilon$ and $\delta_{T_i}\leq 3\epsilon/2$.
If Algorithm~\ref{algo:AGFW-P} stops, it does so after $T_l$ iterations. 
Since $\epsilon\leq \theta+R_g$, we have that $G_t^a \leq \theta+R_g$ at each $(T_i)_{i=1}^l$. By the definition of \texttt{PLMO}, $G_t^a \geq 0$ for all $t$. Therefore, $(\Delta_t)_{t\geq 0}$ is monotonically non-increasing. As such
\[ \Delta_{T_l} \leq \Delta_{T_i} \leq G_{T_i}\quad\textup{for all $1\leq i \leq l$.}\]
Suppose that there is some $1\leq i\leq l$ such that condition~\ref{Cond2} ($\delta_t$-suboptimality) holds at iterate $T_i$. Then  
\[\Delta_{T_i}\leq G_{T_i} \underset{(i)}{\leq} G_{T_i}^a +\delta_{T_i}\leq 5\epsilon/2,\] where (i) follows from Lemma~\ref{lemma:relationGtGta}. Combining these inequalities gives that $\Delta_{T_l}\leq \Delta_{T_i} \leq 5\epsilon/2$. 

Finally, note that the probability that condition~\ref{Cond2} ($\delta_t$-suboptimality) is satisfied for at least one out of $l$ iterates is at least $1-p^l$, so it follows that $\Delta_{T_l}\leq 5\epsilon/2$ with probability at least $1-p^l$.
\end{proof}



\subsection{Convergence Analysis of Algorithm~\ref{algo:AGFW-P}}
The convergence analysis of Algorithm~\ref{algo:AGFW-P}, like that of Algorithm~\ref{algo:AGFW}, also makes use of the partition of the iterates into the $q$-, $r$-, and $s$-subsequences,  
as in Definition~\ref{list:SeqIterates}. If the first $N$ iterates of Algorithm~\ref{algo:AGFW-P} are divided into sequences the $q$-, $r$-, and $s$-subsequences, let $N_r$, $N_s$, and $N_q$ (defined in~\eqref{eqn:DefNrNsNt}) respectively denote the length of each subsequence.

Our analysis of Algorithm~\ref{algo:AGFW-P} follows a broadly similar approach to the analysis of Algorithm~\ref{algo:AGFW}. The key difference is that we need to bound the time it takes to achieve the stopping criterion for the $l$th time, and that the inequality that allows us to control the approximate duality gap along the $q$-subsequence only holds with probability at least $1-p$ independently on each iteration. 

The main technical tool needed to deal with these issues is a 
standard tail bound on the number of successes in a sequence of independent (but not necessarily identically distributed) Bernoulli trials.

\begin{lemma}\label{lemma:ProbGoodIterates}
Let $l$ be a positive integer, let $0<p<1$ and let $0<\overline{p}<1$. Let
$M$ be a positive integer such that $M\geq \frac{l-1+\log(1/\overline{p})}{1-p} + 2\log(1/\overline{p})$. 
Let $\xi_1,\xi_2,\ldots,\xi_M$ be a collection of independent random variables 
such that $\xi_i$ takes value $1$ with probability $1-p_i$ and $\xi_i$ takes value $0$ with probability $p_i$, and assume that $0\leq p_i\leq p$ for all $i$. 
Then the probability that at least $l$ of the variables $\xi_1,\xi_2,\ldots,\xi_M$ takes value $1$ is at least $1-\overline{p}$.
\end{lemma}

\begin{proof}
The number of variables that take value $0$ is $\sum_{i=1}^{M}(1-\xi_i)$, which is a  sum of independent random variables, each taking values in $\{0,1\}$. 
The mean is $\mu = \sum_{i=1}^{M}p_i \leq p M$. By Hoeffding's inequality
\begin{equation}\label{eqn:ProofDistrGoodIterates}
\textup{Pr}\left[\sum_{i=1}^{M}\xi_i \leq (1-p)M - \tau\right] = \textup{Pr}\left[\sum_{i=1}^{M}(1-\xi_i) \geq pM + \tau\right] \leq \textup{Pr}\left[\sum_{i=1}^{M}(1-\xi_i) \geq \mu + \tau\right] \leq \exp \left(-2\tau^2/M\right).
\end{equation}
Let $\tau = \log(1/\overline{p})$. From our assumption that $M\geq 2\log(1/\overline{p}) + (l-1+\log(1/\overline{p}))/(1-p)$ we know that 
$M\geq 2\tau$ and that $M\geq \frac{l-1+\tau}{1-p}$. It follows that $2\tau^2/M \leq \log(1/\overline{p})$ and that $(1-p)M-\tau \geq l-1$. Therefore
\begin{equation}
\textup{Pr}\left[\sum_{i=1}^{M}\xi_i \leq l-1\right] \leq \textup{Pr}\left[\sum_{i=1}^{M}\xi_i \leq (1-p)M-\tau\right]\leq \exp(-2\tau^2/M) \leq \overline{p}.
\end{equation}
\end{proof}

We are now ready to provide bound on number of iterations performed by Algorithm~\ref{algo:AGFW-P} until the stopping criteria (defined in Step 7 of the algorithm) is satisfied. Lemmas~\ref{lemma:ConvProb-OL} and~\ref{lemma:ConvProb-CL} provide a bound on number of iterations performed by Algorithm~\ref{algo:AGFW-P} when $\delta_t$ is set according to scheduled strategy~\ref{Strat1} and adaptive strategy~\ref{Strat2} respectively.

\begin{lemma}\label{lemma:ConvProb-OL}
    Let $0<\epsilon\leq \theta+R_g$ and let $\delta_t$ be a non-negative sequence that converges to $0$. Let $K_q$ be the smallest positive integer such that $\delta_t \leq \epsilon/2$ for all $t\geq K_q$. Let $0< \overline{p}<1$. Let 
    \[ K_r = \lfloor 10.6\Delta_0\rfloor\quad\textup{and}\quad K_s = \left\lceil \frac{48(\theta+R_g)^2}{\epsilon}\right\rceil.\]
    Then, with probability at least $1-\overline{p}$, Algorithm~\ref{algo:AGFW-P} (with~\eqref{eqn:SC-gendef} as input) stops after at most $K_r+(l+1)K_s+2K_q+\frac{l-1+\log(1/\overline{p})}{1-p}+2\log(1/\overline{p})$ iterations.
\end{lemma}


\begin{proof}
For brevity of notation, define $M(l,\overline{p},p) := \frac{l-1\log(1/\overline{p})}{1-p} + 2\log(1/\overline{p})$. 
Let $N = K_r+(l+1)K_s+2K_q+M(l,\overline{p},p)$ and let the first $N$ iterations of Algorithm~\ref{algo:AGFW} be partitioned into the $r$- $s$- and $q$- subsequences according to Definition~\ref{list:SeqIterates} with $C=1/2$. Then, $N = N_r+N_s+N_q$.

From Lemma~\ref{lem:key-ineq-p}, we know that  $N_r \leq K_r$, and so it follows that $N_s +N_q \geq ((l+1)K_s+K_q)+\left(K_q+M(l,\overline{p},p)\right)$. We now consider two cases: (i) $N_s \geq (l+1)K_s+K_q+1$, and (ii) $N_s \leq (l+1)K_s+K_q$. In the first case, the $s$-subseqence will be long enough to establish that the algorithm reaches the stopping criterion. In the second case, the $q$-subsequence will be long enough to ensure that the algorithm reaches the stopping criterion with probability at least $1-\overline{p}$.

\textbf{Case 1: $N_s \geq  (l+1)K_s+K_q+1$.} 
For $u=1,2,\ldots,l$, let 
\[ k_u^* = \argmin_{i\in \{uK_s+K_q+1,\ldots,(u+1)K_s+K_q+1\}} G_{s_i}^a.\]
Since $N_s\geq (l+1)K_s+K_q+1$ it follows that $s_{k_u^*} \leq s_{(l+1)K_s+K_q+1} \leq s_{N_s} \leq N$, for all $u=1,2,\ldots,l$, as required. We now establish bounds on $\delta_{s_{k_u^*}}, G_{s_{k_u^*}}^a$, for $u=1,2,\ldots,l$, showing that at each of these iterations, the stopping criterion is satisfied.
\begin{itemize}
    \item Since $k_u^*\geq K_q$, it follows that $s_{k_u^*} \geq s_{K_q} \geq K_q$, and so that $\delta_{s_{k_u^*}}\leq \epsilon/2$ (by the definition of $K_q$).
    \item By the definition of $k_u^*$, from Lemma~\ref{lem:key-ineq-p} (with $l=uK_s+K_q$ and $j=(u+1)K_s+K_q$) we have that $l\geq \lfloor j/2+1\rfloor$ and so 
\[ G_{s_{k_u^*}}^a \leq \frac{24(\theta+R_g)^2}{\sqrt{(uK_s+K_q+2)K_s}}\leq \epsilon/2.\]
\end{itemize}
It follows that Algorithm~\ref{algo:AGFW-P} terminates at iteration $s_{k_l^*}\leq N$.

\textbf{Case 2: $N_s \leq  (l+1)K_s+K_q$.} Since $N_s \leq  (l+1)K_s+K_q$, we have that $N_q \geq (l+1)K_s+2K_q +M(l,\overline{p},p) - N_s\geq K_q+M(l,\overline{p},p)$. We will show that among the iterates $t\in \{q_{K_q},q_{K_q+1},\ldots,q_{N_q}\}$, the conditions $\delta_t\leq \epsilon/2$ and $G_t^a\leq \epsilon/2$ hold at least $l$ times with probability at least $1-\overline{p}$. This ensures that Algorithm~\ref{algo:AGFW-P} stops with probability at least $1-\overline{p}$ by iterate $q_{N_q} \leq N$. 

Since $q_{K_q}\geq K_q$, it follows (from the definition of $K_q$) that $\delta_{q_k}\leq \epsilon/2$ for $k=K_q,K_q+1,\ldots,N_q$.
The subsequence $q_k$ for $k\in\{K_q,K_q+1,\ldots,N_q\}$, has length at least $M(l,\overline{p},p)$.  
 From Lemma~\ref{lem:key-ineq-p}, we know that $G_{q_k}^a \leq \delta_{q_k} \leq \epsilon/2$ holds with probability at least $1-p$, independently for each $k\in\{K_q,K_q+1,\ldots,N_q\}$. From Lemma~\ref{lemma:ProbGoodIterates}, we can conclude that, with probability at least $1-\overline{p}$, we have that $G_{q_k}^a \leq \epsilon/2$ holds at least $l$ times over the iterates $k\in\{K_q,K_q+1,\ldots,N_q\}$.

\end{proof}

We now consider the case when $\delta_t$ is set via the adaptive strategy~\ref{Strat2}. 
Recall that in the deterministic case, the main idea was that the approximate duality gap decreases geometrically along the $q$-subsequence. In the random setting, this does not quite hold, because the bound $G_{q_k}^a \leq \frac{1}{2C}\delta_{q_k}$ fails to hold with probability at most $p$. However, if the $q$-subsequence is long enough, then this relationship does hold along a sufficiently long subsequence of the $q$-subsequence, so that the stopping criterion is reached (with any desired probability). 

\begin{lemma}\label{lemma:ConvProb-CL}  
    Let $0<\epsilon\leq \theta + R_g$, and let $\delta_t$ be a non-negative sequence such that $\delta_t \leq \epsilon/2 + \min_{\tau<t}G_\tau^a$ for all $t$. Let $0<\overline{p}<1$. Let  
    \[ K_r = \lfloor 10.6\Delta_0\rfloor\quad\textup{and}\quad K_s = \left\lceil \frac{72(\theta+R_g)^2}{\epsilon}\right\rceil\quad\textup{and}\quad K_q = 2+\left\lceil \log_2\left(\frac{\theta+R_g}{\epsilon}\right)\right\rceil.\]
    Then, with probability at least $1-\overline{p}$, Algorithm~\ref{algo:AGFW-P} (with~\eqref{eqn:SC-gendef} as input)
    stops after at most 
    $K_r+(l+1)K_s+\frac{K_q+l-1+\log(1/\overline{p})}{1-p} + 2 \log(1/\overline{p})$ iterations.
\end{lemma}

\begin{proof}
Let $N = K_r+(l+1)K_s+M(K_q,l,\overline{p},p)$ where $M(K_q,l,\overline{p},p) = \frac{K_q+l-1+\log(1/\overline{p})}{1-p} + 2 \log(1/\overline{p})$. Let the first $N$ iterations of Algorithm~\ref{algo:AGFW} be partitioned into the $r$-, $s$-, and $q$-subsequences according to Definition~\ref{list:SeqIterates} with $C=1$. Then $N = N_r+N_s+N_q$. We know that $N_r \leq K_r$ (from~\ref{lem:key-ineq-p}). It follows that $N_s+N_q \geq (l+1)K_s + M(K_q,l,\overline{p},p)$. We now consider two cases: (i) $N_q \leq M(K_q,l,\overline{p},p)-1$, (ii) $N_q\geq  M(K_q,l,\overline{p},p)$.

\textbf{Case 1: $N_q \leq M(K_q,l,\overline{p},p)-1$.} If $N_q \leq M(K_q,l,\overline{p},p)-1$ then $N_s \geq (l+1)K_s+1$. For $u=1,2,\ldots,l$, let 
\[ k_u^* = \argmin_{i\in \{uK_s+1,\ldots,(u+1)K_s+1\}} G_{s_i}^a.\]
Since $N_s \geq (l+1)K_s+1$, it follows that $s_{k_u^*} \leq s_{(l+1)K_s+1} \leq s_{N_s}\leq N$, for all $i=1,2,\ldots,l$, as required. We now establish bounds on $\delta_{s_{k_u^*}}$ and $G^a_{s_{k_u^*}}$, for $u=1,2,\ldots,l$, showing that at each of these $l$ iterations, the stopping criterion is satisfied. 
\begin{itemize} 
\item Since $k^*_u \geq uK_s+1$, from Lemma~\ref{lem:key-ineq-p} (with $l=0$ and $j=uK_s+1$) we have that  
    \[\delta_{s_{k_u^*}} \leq \epsilon/2 + \min_{\tau < s_{k_u^*}} G_\tau^a \leq \epsilon/2 + \min_{k\in \{1,2,\ldots,uK_s\}} G_{s_k}^a \leq \epsilon/2+\frac{72(\theta+R_g)^2}{uK_s} \leq 3\epsilon/2,\]
    for $u=1,2,\ldots,l$.
    \item Applying Lemma~\ref{lem:key-ineq-p} (with $l=K_s$ and $j=2K_s$), 
    we have that $l = j/2<\lfloor j/2+1\rfloor$, and so
    \[ G_{s_{k_1^*}}^a \leq \frac{72(\theta+R_g)^2}{2K_s+1} \leq \epsilon.\]
    For $u=2,3,\ldots,l$, we again apply Lemma~\ref{lem:key-ineq-p} (with $l=uK_s$ and $j=(u+1)K_s$). In these cases, $l\geq \lfloor j/2+1\rfloor$ and so
    \[ G_{s_{k_u^*}}^a \leq \frac{36(\theta+R_g)^2}{\sqrt{(uK_s+2)K_s}} \leq \epsilon\]
    for $u=2,3,\ldots,l$.
    
\end{itemize}

\textbf{Case 2: $N_q \geq M(K_q,l,\overline{p},p)$.}
It is convenient to define a further subsequence of the $q$-subsequence. Let the $v$-subsequence $v_1<v_2<\cdots$ consist of those $t\in\mathbb{N}$ for which $t = q_k$ for some $k$ and $G_t^a \leq \delta_{t}/2$. The $v$-subsequence consists of those indices from the $q$-subseqence at which the \texttt{PLMO} satisfies condition~\ref{Cond2} ($\delta_t$-suboptimality). From Lemma~\ref{lem:key-ineq-p}, we know that the $v$-subsequence can be obtained from the $q$-subsequence by (independently) keeping each element of the $q$-subsequence with some probability that is at least $1-p$. 
Let $N_v$ denote the (random) number of elements of $\{1,2,\ldots,N\}$ that are in the $v$-subsequence. From Lemma~\ref{lemma:ProbGoodIterates}, we know that if $N_q \geq \frac{l+K_q-1+\log(1/\overline{p})}{1-p} + 2\log(1/\overline{p})$ then $N_v \geq l+K_q$ with probability at least $1-\overline{p}$. 

By the same argument as in Case 2 of Lemma~\ref{lemma:ConvDet-CL} (but replacing $q_{(\cdot)}$ with $v_{(\cdot)}$) we have that 
\[ \delta_{v_k} \leq \epsilon/2 + G_{v_{k-1}}^a\]
for $k\geq 2$ and that 
\[ G_{v_k}^a \leq \epsilon/2 + 2^{-k+1}(\theta+R_g)\]
for $k\geq 2$. 
Then, if $k\geq K_q+1$, we have that 
$G_{v_k}^a \leq 3\epsilon/4$ and that $\delta_{v_k} \leq 3\epsilon/2$.
In particular, the stopping criterion of Algorithm~\ref{algo:AGFW-P} holds at iterates $v_{K_q+1},\ldots,v_{K_q+l}$. As long as $v_{K_q+l}\leq N$, the stopping criterion of Algorithm~\ref{algo:AGFW-P} is achieved $l$ times. We have seen that, with probability at least $1-\overline{p}$, we have that $l+K_q \leq N_v$ and so $v_{l+K_q}\leq v_{N_v} \leq N$, as required.  
\end{proof}

\section{Application of Algorithm~\ref{algo:AGFW-P} to~\eqref{prob:SC-SumLogDef}}\label{sec:SC-Barrier}
In this section, we consider the application of Algorithm~\ref{algo:AGFW-P} to generate a near-optimal solution to~\eqref{prob:SC-SumLogDef}. The aim of this section is to work through this example as it covers the study of motivating example introduced in Section~\ref{sec:Intro-Issues}. We show how to implement \texttt{PLMO} using Lanczos method in Section~\ref{sec:PLMO-Implementation}, followed by stating the computational complexity of Algorithm~\ref{algo:AGFW-P} in Section~\ref{sec:TotalConvSCBarrier}, where we explicitly define all the parameters of the problem. Finally, in Section~\ref{sec:Memory-Efficient}, we show that Algorithm~\ref{algo:AGFW-P} can be made memory-efficient (requiring $\mathcal{O}(n+d)$ memory) using memory-efficient techniques proposed by~\citet{shinde2021memory}.
We first restate this problem as the composite optimization problem
\begin{equation}\label{prob:SC-SumLogDefComp}\tag{SC-GMean-Cmp}
\min_{X\in \mathbb{S}^n} -\sum_{i=1}^d \log(\langle A_i,X\rangle) +g(X),
\end{equation}
where $g(X)$ is the indicator function of $\mathcal{S} = \{ X\succeq 0: \textup{Tr}(X) = 1\}$ and $f(\mathcal{B}(X))= -\sum_{i=1}^d \log(\langle A_i,X\rangle)$ is a $d$-logarithmically-homogeneous,
2-self-concordant function defined on $\textup{int}(\mathcal{S})$. In the case of~\eqref{prob:SC-SumLogDefComp}, the linear subproblem~\eqref{prob:LinearMinimizationProblem} takes the form
\begin{equation}\label{prob:SubProb-SCBarrier}
\min_{H\in\mathcal{S}}\langle \mathcal{B}^*(\nabla f(\mathcal{B}(X^t))), H\rangle \equiv \min_{H\in\mathcal{S}}\left\langle -\sum_{i=1}^d \frac{A_i}{\langle A_i,X^t\rangle}, H\right\rangle
\end{equation}
 which is equivalent to solving an eigenvalue problem. Indeed, if $u^{\star}$ is a unit norm eigenvector corresponding to the smallest eigenvalue of $\mathcal{B}^*(\nabla f(\mathcal{B}(X)))$, then $H^\star = u^{\star}u^{\star T}$ is an optimal solution to the linear subproblem~\eqref{prob:SubProb-SCBarrier}. Now, to generate the update direction $H^t$ in Step 6 of Algorithm~\ref{algo:AGFW-P}, we need to find an approximate solution to the linear subproblem or equivalently an approximate solution to the corresponding extreme eigenvalue problem. In this paper, we consider using the Lanczos method (with random initialization) to generate such an approximate solution. In the next subsection, we review an existing convergence result for the Lanczos method and show how it can ultimately be used to 
 implement \texttt{PLMO}. Furthermore, in Lemma~\ref{lemma:LMOImplementationBarrier}, we provide a bound on number of iterations of the Lanczos method needed to implement \texttt{PLMO}. Finally, in Subsection~\ref{sec:TotalConvSCBarrier}, we provide the `inner' iteration complexity, i.e., the computational complexity of each iteration of Algorithm~\ref{algo:AGFW-P} applied to~\eqref{prob:SC-SumLogDefComp}, as well as a bound on the total computational complexity of the algorithm. 

\subsection{Implementing \texttt{PLMO} in Algorithm~\ref{algo:AGFW-P} when applied to~\eqref{prob:SC-SumLogDefComp}}\label{sec:PLMO-Implementation}
The Lanczos method (see, e.g.,~\cite{kuczynski1992estimating}) is an iterative method that can be used to determine the extreme eigenvalues and eigenvectors of a Hermitian matrix. A typical convergence result for the method, when initialized with a uniform random point on the unit sphere, is given in Lemma~\ref{lemma:ConvLanczos}.

\begin{lemma}[Convergence of Lanczos method \cite{kuczynski1992estimating}]\label{lemma:ConvLanczos}
Let $\rho \in (0,1]$ and $p \in (0,1]$. For $J\in \mathbb{S}^n$, let $\|J\|$ denote the largest absolute eigenvalue of $J$. The Lanczos method,
after at most $N=\left\lceil \frac{1}{2} + \frac{1}{\sqrt{\rho}}\log(4n/p^2)\right\rceil$ iterations, generates a unit vector $u\in \mathbb{R}^n$, that satisfies
\begin{equation}\label{eqn:LancProofConv}
 u^T Ju \geq \lambda_{\textup{max}}(J) - \frac{\rho}{8}\|J\|
\end{equation}
with probability at least $1-p$. The computational complexity of the method is $\mathcal{O}(N(n+\textup{mvc}(J)))$, where $\textup{mvc}(J)$ is the matrix-vector multiplication complexity for the matrix $J$. 
\end{lemma}

Algorithm~\ref{algo:PLMO-Barrier} shows how to use the Lanczos method as the basis for an implementation of \texttt{PLMO}, to generate the update direction in Algorithm~\ref{algo:AGFW-P}. The following result establishes that Algorithm~\ref{algo:PLMO-Barrier} is indeed a valid implementation of \texttt{PLMO}.

\begin{algorithm}[h]
\caption{\texttt{PLMO} for~\eqref{prob:SC-SumLogDefComp}}
\label{algo:PLMO-Barrier}
\textbf{Input}: $\mathcal{B}^{\star}(\nabla f(\mathcal{B}(X^t)))$, $\delta_t$, $X^t$
\textbf{Output}: ($H^t$, $G_t^a$)
\begin{algorithmic}[1] 
\STATE Set $J=-\mathcal{B}^*(\nabla f(\mathcal{B}(X^t)))\succeq 0$.
\STATE Run Lanczos method for $N = \frac{1}{2}+\sqrt{\frac{c\theta}{8\min\{\delta_t,(c-2)\theta\}}}\log(4n/p^2)$ iterations to generate a unit vector $u$.
\STATE Set $H^t = uu^T$.
\STATE Compute $G_t^a(X^t, H^t)$ according to~\eqref{eqn:DefGkApprox}.
\IF {$G_t^a < 0$} \STATE Set $H^t = X^t$ and $G_t^a = 0$.
\ENDIF
\STATE \textbf{return} ($H^t$, $G_t^a$)
\end{algorithmic}
\end{algorithm}

\begin{lemma}\label{lemma:LMOImplementationBarrier}
Algorithm~\ref{algo:PLMO-Barrier} implements \texttt{PLMO} correctly for the problem~\eqref{prob:SC-SumLogDefComp}.
\end{lemma}

\begin{proof}
For $J=-\mathcal{B}^*(\nabla f(\mathcal{B}(X^t)))\succeq 0$, we can rewrite~\eqref{eqn:LancProofConv} as $u^TJu\geq \lambda_{\textup{max}}(J)(1-\rho/8)=\|J\|(1-\rho/8)$. Let $\rho = 8\tau_t = 8\frac{\textup{min}\{\delta_t,(c-2)\theta\}}{c\theta}$ (set according to Lemma~\ref{lemma:LMODef-RelErr}).
From Lemma~\ref{lemma:ConvLanczos}, we see that after $N=\frac{1}{2}+\sqrt{\frac{c\theta}{8\min\{\delta_t,(c-2)\theta\}}}\log(4n/p^2)$ iterations, the Lanczos method generates a unit vector $u$ that satisfies, with probability at least $1-p$,
\begin{equation}\label{eqn:AltProof-SCB}
u^TJu \geq \|J\| \left(1-\frac{\min\{\delta_t,(c-2)\theta\}}{c\theta}\right) = \|J\|(1-\tau_t),
\end{equation}
where the last equality follows from the definition of $\tau_t$.

Furthermore, since $H^t = uu^t$, and $J=-\mathcal{B}^*(\nabla f(\mathcal{B}(X^t)))$, we can rewrite~\eqref{eqn:AltProof-SCB} as
\begin{equation}\label{eqn:AltProof-SCB1}
-\langle J, H^t\rangle = F^{\textup{lin}}_{X^t}(H^t) \leq -\|J\|(1-\tau_t) = (1-\tau_t)F^{\textup{lin}}_{X^t}(H^\star),
\end{equation}
where the last equality follows since $\|J\| =  -\lambda_{\textup{min}}(\mathcal{B}^*(\nabla f_t)) = -F^{\textup{lin}}_{X^t}(H^\star)$. So, from Lemma~\ref{lemma:LMODef-RelErr}, we see that, if~\eqref{eqn:AltProof-SCB1} holds, then at least one of the two conditions~\ref{Cond1} (large gap) and~\ref{Cond2} ($\delta_t$-suboptimality) is satisfied. And since~\eqref{eqn:AltProof-SCB1} holds with probability at least $1-p$, we have that after at most $N$ iterations of Lanczos method, the generated pair $(H^t,G_t^a)$ satisfy at least one of the two conditions~\ref{Cond1} (large gap) and~\ref{Cond2} ($\delta_t$-suboptimality) with probability at least $1-p$.

Finally, if condition~\ref{Cond2} ($\delta_t$-suboptimality) is satisfied and yet $G_t^a<0$ then, from Lemma~\ref{lem:Gta0}, we know that $H^t = X^t$ also satisfies condition~\ref{Cond2} ($\delta_t$-suboptimality) with corresponding value of $G_t^a$ equal to $0$.
\end{proof}

We are now ready to provide the convergence analysis and the total computational complexity of Algorithm~\ref{algo:AGFW-P} when applied to~\eqref{prob:SC-SumLogDefComp} in the next subsection.

\subsection{Computational complexity of Algorithm~\ref{algo:AGFW-P} for~\eqref{prob:SC-SumLogDef}}\label{sec:TotalConvSCBarrier}
In this section, we provide the computational complexity of Algorithm~\ref{algo:AGFW-P} when applied to~\eqref{prob:SC-SumLogDef}.
The total computational complexity of Algorithm~\ref{algo:AGFW-P} consists of the `inner' iteration complexity, i.e., the computational complexity of each iteration of Algorithm~\ref{algo:AGFW-P} multiplied by the number of iterations performed by the algorithm.
The two most computationally expensive steps in Algorithm~\ref{algo:AGFW-P} are (i) Step 10 (computing $D_t$), 
and (ii) Step 6, implementing \texttt{PLMO}, whose computational complexity can be determined using Lemmas~\ref{lemma:ConvLanczos} and~\ref{lemma:LMOImplementationBarrier}. We elaborate on this complexity result in the proof of Lemma~\ref{lemma:TotalConvResultSCBarrier}.

Lemmas~\ref{lemma:ConvProb-OL} and~\ref{lemma:ConvProb-CL} (from Section~\ref{sec:ProbLMO}) provide a bound on the number of iterations performed by Algorithm~\ref{algo:AGFW-P} when the scheduled strategy~\ref{Strat1} and the adaptive strategy~\ref{Strat2} are used, respectively. Note that the bounds given in these lemmas depend on the parameters $\theta$, $R_g$ and $\Delta_0$. We first compute a bound on these quantities and finally provide the total computational complexity of Algorithm~\ref{algo:AGFW-P} in Lemma~\ref{lemma:TotalConvResultSCBarrier}.

\paragraph{Values of parameters $\theta$, $R_g$ and $\Delta_0$.} Since $f(\mathcal{B}(X))$ in~\eqref{prob:SC-SumLogDefComp} is a $d$-logarithmically-homogeneous function, we have $\theta = d$. Also, since $g(\cdot)$ is an indicator function for the set $\mathcal{S}$ and Algorithm~\ref{algo:AGFW-P} generates feasible updates to the decision variable, we have $R_g = 0$.

Finally, we determine an upper bound on $\Delta_0$. The initial optimality gap, $\Delta_0$ is defined as \begin{equation}
\Delta_0 =  f(\mathcal{B}(X^0)) - f(\mathcal{B}(X^{\star})) + g(X^0) - g(X^{\star}) = f(\mathcal{B}(X^0)) - f(\mathcal{B}(X^{\star})),
\end{equation}
where the last equality holds since 
$X^{\star}\in \mathcal{S}$ and $X^0\in \mathcal{S}$. Now let $X^0 = I/n$, so that $f(\mathcal{B}(I/n)) = -\sum_{i=1}^d \log(\textup{Tr}(A_i)) + d\log(n)$.
To generate a lower bound on $f(\mathcal{B}(X^{\star}))$, we use the knowledge of the dual function. The dual of~\eqref{prob:SC-SumLogDef} is
\begin{equation}\label{prob:SCLogDual}
\max_{y, \lambda}\ -\lambda + \sum_{i=1}^d (\log(y_i) + 1)\ \ \textup{s.t.}\ \begin{cases}
& \sum_{i=1}^d A_i y_i \preceq \lambda I\\
&y_i \geq 0,\quad i = 1,\dotsc, d.
\end{cases} 
\end{equation}
Let $\overline{\lambda} = d$ and $\overline{y}=1/\textup{Tr}(A_i)$, for $i = 1,\dotsc,d$. Note that, $(\overline{\lambda}$, $\overline{y})$ is a feasible solution to the dual problem~\eqref{prob:SCLogDual}. Therefore,
\begin{equation}
f(\mathcal{B}(X^{\star})) \geq -\overline{\lambda} + \sum_{i=1}^d (\log(\overline{y}_i) +1) = -\sum_{i=1}^d \log(\textup{Tr}(A_i)).
\end{equation}
The optimality gap $\Delta_0$ can now be bounded by
\begin{equation}\label{eqn:BoundDelta0}
\Delta_0 =  f(\mathcal{B}(X^0)) - f(\mathcal{B}(X^{\star})) \leq d\log(n).
\end{equation}

We now provide the convergence analysis of Algorithm~\ref{algo:AGFW-P} in Lemma~\ref{lemma:TotalConvResultSCBarrier}.

\begin{lemma}\label{lemma:TotalConvResultSCBarrier}
Let~\eqref{prob:SC-SumLogDefComp} be solved using Algorithm~\ref{algo:AGFW-P} with parameters $l$, $\epsilon$ and $p$, such that \texttt{PLMO} fails with probability at most $p$.
\begin{enumerate}
\item \textbf{Outer iteration complexity when scheduled strategy~\ref{Strat1} with $K_q = 1$ is used.} In this case, with probability at least $1-\overline{p}$, Algorithm~\ref{algo:AGFW-P} stops after at most \[\overline{K} =  \lceil 10.6d\log n \rceil +(l+1)\left\lceil \frac{48d^2}{\epsilon}\right\rceil+2+\frac{l-1+\log(1/\overline{p})}{1-p}+2\log(1/\overline{p})
\]
iterations.
\item \textbf{Inner iteration complexity when scheduled strategy~\ref{Strat1} is used.} When scheduled strategy~\ref{Strat1} with $K_q = 1$ and $\delta_t = \epsilon/2$ for all $t\geq 1$ is used in Algorithm~\ref{algo:AGFW-P}, and Lanczos method is used to implement \texttt{PLMO} in each iteration of the algorithm, then the total computational complexity of each iteration of Algorithm~\ref{algo:AGFW-P} is bounded by $\mathcal{O}\left(  \sqrt{\frac{d}{\epsilon}}\log (4n/p^2)(n+\sum_{i=1}^d\textup{mvc}(A_i))\right)$, where $\textup{mvc}(A_i)$ denotes the complexity of matrix-vector multiplication with $A_i$.
\item \textbf{Total computational complexity when scheduled strategy~\ref{Strat1} is used.} When scheduled strategy~\ref{Strat1} with $K_q = 1$ and $\delta_t = \epsilon/2$ for all $t\geq 1$ is used in Algorithm~\ref{algo:AGFW-P}, the total computational complexity of the algorithm is bounded by $\mathcal{O}\left(\sqrt{\frac{d}{\epsilon}}\log n \left(d\log n+\frac{ld^2}{\epsilon}\right)\left( n+\sum_{i=1}^d\textup{mvc}(A_i)\right)\right)$.
\end{enumerate}
\end{lemma}

\begin{proof}
\textbf{Outer iteration complexity.}  When scheduled strategy~\ref{Strat1} with $K_q = 1$ is used, we have that $\delta_t \leq \epsilon/2$ for all $t\geq 1$. Substituting $K_q = 1$, and values of parameters  $\theta$, $R_g$ and $\Delta_0$ in Lemma~\ref{lemma:ConvProb-OL}, we see that, with probability at most $1-\overline{p}$, Algorithm~\ref{algo:AGFW-P} stops after at most \[\overline{K} = \lceil 10.6d\log n \rceil +(l+1)\left\lceil \frac{48d^2}{\epsilon}\right\rceil+2+\frac{l-1+\log(1/\overline{p})}{1-p}+2\log(1/\overline{p})\]
iterations. Thus, we have that $\overline{K}$ is bounded by $\mathcal{O}\left(d\log n+\frac{ld^2}{\epsilon} \right)$.

\textbf{Inner iteration complexity when scheduled strategy~\ref{Strat1} is used.} The two most expensive steps in Algorithm~\ref{algo:AGFW-P} are computing $D_t$, and implementing \texttt{PLMO}.
\texttt{PLMO} is implemented using Lanczos method whose computational complexity is given in Lemma~\ref{lemma:ConvLanczos}.
Assume that scheduled strategy~\ref{Strat1} with $K_q = 1$, and $\delta_t = \epsilon/2$ is used for all $t\geq 1$.
From Lemma~\ref{lemma:ConvLanczos}, we see that the computational complexity of the Lanczos method depends on $\textup{mvc}(\mathcal{B}^{\star}(\nabla f(\mathcal{B}(X^t)))$. Note that, $\mathcal{B}^{\star}(\nabla f(\mathcal{B}(X^t)) = -\sum_{i=1}^d\frac{A_i}{\langle A_i,X^t\rangle}$. If $X^0 = \frac{1}{n}I$, then $\langle A_i, X^0\rangle = \textup{Tr}(A_i)$ for all $i=1,\dotsc,d$. Moreover, we have that $X^{t} = (1-\gamma_{t-1})X^{(t-1)} + \gamma_{t-1}H^{(t-1)}$ and $\textup{rank}(H^{(t-1)}) \leq 1$ for all $t\geq 2$. Thus, $\langle A_i, X^t\rangle = \langle A_i, X^{(t-1)}\rangle + u^TA_iu$, where $H^t = uu^T$. And so, if we store $\langle A_i, X^{(t-1)}\rangle$ for all $i$, then we see that the complexity of computing $\langle A_i,X^t\rangle$ for all $i\in 1,\dotsc,d$ is bounded by $\sum_{i=1}^d\textup{mvc}(A_i)$, i.e., the complexity is bounded by the sum of complexities of matrix-vector multiplication with each $A_i$. Finally, from Lemmas~\ref{lemma:ConvLanczos} and~\ref{lemma:LMOImplementationBarrier} (with $c=4$), we have that computational complexity of \texttt{PLMO} at iteration $t$ of the Algorithm~\ref{algo:AGFW-P} is bounded by $\mathcal{O}\left(\sqrt{\frac{d}{\epsilon}}\log (4n/p^2)(n+\sum_{i=1}^d\textup{mvc}(A_i))\right)$. 

Furthermore, recall that $D_t = \|H^t-X^t\|_{\mathcal{B}(X_t)} = \sqrt{\sum_{i=1}^d\frac{\langle A_i, H^t-X^t\rangle^2}{\langle A_i,X^t\rangle^2}}$, and so, the complexity of computing $D_t$ is also bounded by $\mathcal{O}(\sum_{i=1}^d\textup{mvc}(A_i))$.
Thus, the total computational complexity of each iteration of Algorithm~\ref{algo:AGFW-P} is bounded by $\mathcal{O}\left(  \sqrt{\frac{d}{\epsilon}}\log (4n/p^2)(n+\sum_{i=1}^d\textup{mvc}(A_i))\right)$.

\textbf{Total computational complexity when scheduled strategy~\ref{Strat1} is used.} Note that, Algorithm~\ref{algo:AGFW-P} stops after at most $\overline{K}$ iterations with probability at least $1-\overline{p}$. 
Thus, the total computational complexity of Algorithm~\ref{algo:AGFW-P} is 
$\mathcal{O}\left(\left(\sqrt{\frac{d}{\epsilon}}\log (4n/p^2)(n+\sum_{i=1}^d\textup{mvc}(A_i))\right)\overline{K}\right)$ or equivalently, \[\mathcal{O}\left(\sqrt{\frac{d}{\epsilon}}\log n \left(d\log n+\frac{ld^2}{\epsilon}\right)\left( n+\sum_{i=1}^d\textup{mvc}(A_i)\right)\right).\]
\end{proof}

\subsection{Making Algorithm~\ref{algo:AGFW-P} memory-efficient}\label{sec:Memory-Efficient}
One reason to use methods based on the Frank-Wolfe algorithm for semidefinite optimization problems with constraint $\mathcal{S}$, is that they typically can be modified to lead to approaches that do not require explicit storage of a PSD decision variable, and hence can be made memory-efficient. One approach to this is to represent the PSD decision variable $X$ by a Gaussian random vector $z\sim \mathcal{N}(0,X)$, as suggested by~\citet{shinde2021memory}. This is particularly useful in situations where the semidefinite program is used as the input to a rounding scheme that only requires such Gaussian samples. An example of this situation is discussed in Section~\ref{sec:Intro-OptQuadForm}, where an approximation algorithm for maximizing the product of quadratic forms on the unit sphere is obtained by solving~\eqref{prob:SC-SumLogDefComp} to get a matrix $X\in \mathcal{S}$, and then sampling a Gaussian vector $z$ with covariance equal to $X$ and normalizing it, to get a feasible point for the original optimization problem on the sphere.

The Gaussian sampling-based approach takes advantage of two properties of a problem with the form 
\[ \min_{X\in \mathcal{S}} f(\mathcal{B}(X)).\]
The first is the fact that the objective function only depends on $X^t$ via the lower dimensional vector $v^t := \mathcal{B}(X^t)$ with coordinates $v_i^t = \langle A_i,X^t\rangle$ for $i\in [d]$. 
The second key property is the fact that the linearized subproblem is an extreme eigenvalue problem. Therefore it always has a rank one optimal point, and we can implement \texttt{PLMO} so that the output $H^t$ is always rank one.
The Gaussian-sampling based approach then involves making the following modifications to the classical Frank-Wolfe algorithm~\cite[Algorithm~2]{shinde2021memory}:
\begin{itemize}
    \item Initializing the samples with $z^0\sim\mathcal{N}(0,X^0)$, where $X^0$ is a covariance matrix for which the samples $z^0$ can be generated without explicitly storing $X_0$, such as $X^0 = I/n$. 
    \item Modifying \texttt{PLMO} so that it returns the vector $u^t$ rather than the rank one matrix $H^t = (u^t)(u^t)^T$.
    \item Updating the samples via $z^{t+1} = \sqrt{1-\gamma_t}z^t + \sqrt{\gamma_t}\zeta_t u^t$ where the $\zeta_t$ are a sequence of i.i.d.~$\mathcal{N}(0,1)$ random variables.
    \item Updating $v^t := \mathcal{B}(X^t)$ via $v^{t+1}_i = (1-\gamma_t)v^t_i + \gamma_t (u^t)^TA_i u^t$ for $i\in [d]$
    \item Expressing the linearized objective function in terms of $v^t$ via  $\mathcal{B}^\star(\nabla f(\mathcal{B}(X^t))) = \sum_{i=1}^{d}A_i \partial_i f(v^t)$,
    where $\partial_i f$ is the partial derivative of $f$ with respect to the $i$th coordinate.
\end{itemize}
The updates for the modified algorithm only depend on the parameters $A_i$ of the linear map $\mathcal{B}$ via matrix-vector products with the $A_i$. To run the algorithm, only $z$ (a vector of length $n$) and $v$ (a vector of length $d$) need to be stored. If the \texttt{PLMO} is implemented via an algorithm that only requires matrix-vector multiplications with $\mathcal{B}^\star(\nabla f(\mathcal{B}(X^t))$ (such as the power method), then this only requires $v^t$ and access to the $A_i$ via matrix-vector products, and only requires $\mathcal{O}(n)$ storage (in addition to the storage associated with the $A_i$).  


When extending this approach to the setting of generalized Frank-Wolfe, the main additional issue is computing $D_t = \|H^t - X^t\|_{\mathcal{B}(X^t)}$ and $G_t^a$ (and consequently $\gamma_t$). These computations can be done in a memory-efficient way because
\[ D_t = \|H^t - X^t\|_{\mathcal{B}(X^t)} = \sqrt{(w^t-v^t)^T\nabla^2 f(v^t)(w^t-v^t)}\quad\textup{and}\quad G_t^a = \langle v^t-w^t,\nabla f(v^t)\rangle,\]
where $w^t = \mathcal{B}(H^t)$ is the vector with coordinates $w^t_i = (u^t)^T A_i (u^t)$ for $i\in [d]$. Again, computing $G_t^a$ and $D_t$ only requires the vector $v^t$ as well as access to the $A^i$ via matrix-vector multiplications to compute $w^t$.

In the specific case of~\eqref{prob:SC-SumLogDefComp}, we have that $f$ is separable and so 
\begin{equation*}
    \mathcal{B}^\star(\nabla f(v^t)) = -\sum_{i=1}^{d}A_i/v^t_i\quad\textup{and}\quad
    D_t  = \sqrt{\sum_{i=1}^{d} (w_i^t/v_i^t - 1)^2}\quad\text{and}\quad G_t^a = \sum_{i=1}^{d}(w^t_i/v^t_i - 1). 
\end{equation*}

\begin{prop}
Suppose that the power method~\cite{kuczynski1992estimating} 
is used to implement \texttt{PLMO} in Algorithm~\ref{algo:AGFW-P}. 
Then, Algorithm~\ref{algo:AGFW-P}, combined with the Gaussian sampling-based representation of the decision variable, uses $\mathcal{O}(n+d)$ memory (in addition to the memory required to store the input parameters) to generate a zero-mean Gaussian random vector $\widehat{z}_{\epsilon}\sim \mathcal{N}(0,\widehat{X}_{\epsilon})$ such that $\widehat{X}_{\epsilon}$ is an $\epsilon$-optimal solution to the input problem~\eqref{prob:SC-SumLogDefComp}.
\end{prop}

\section{Numerical Experiments}\label{sec:NE}
For our experiments, we implemented Algorithm~\ref{algo:AGFW-P} with two different strategies~\ref{Strat1} (scheduled) and~\ref{Strat2} (adaptive) for $\delta_t$, as well as \cite[Algorithm~1]{zhao2022analysis}. We generated two types of random instances of~\eqref{prob:SC-SumLogDef}, and reported the output of the three algorithms described below. 
The aim of our experiments was to check the performance of Algorithm~\ref{algo:AGFW-P}, and investigate how conservative the theoretical results are against the practical observations.

\paragraph{Generating random instances of~\eqref{prob:SC-SumLogDefComp}.}
We generated two types of problem instances of~\eqref{prob:SC-SumLogDefComp}, namely \textsc{Diag}, \textsc{Rnd}.

\begin{itemize}
\item \textbf{\textsc{Diag} problem instances.}  In \textsc{Diag} problem instances, $A_i\in\mathbb{S}^n$ were generated as follows: $[A_i]_{ii} = i$ for $i=1,\dotsc,d$, with other entries in the matrix set to 0. For \textsc{Diag} instances, the optimal solution $X^{\star}$ has a $d\times d$ scaled identity matrix $I/d$ on its principal diagonal and 0 everywhere else. Accordingly, we computed the optimality gap $\Delta_0 = F(X^0) - F(X^{\star})$.
\item \textbf{\textsc{Rnd} problem instances.} \textsc{Rnd} problem instances were generated by generating parameter matrices $A_i$'s for $i=1,\dotsc,d$ such that $A_i = \sum_{j=1}^n u_ju_j^T$ where $u_j \sim \mathcal{N}(0,I)$ for each $j\in [n]$. For these instances, we use $d\log n$ as an upper bound on $\Delta_0$ as derived in Section~\ref{sec:TotalConvSCBarrier}.
\end{itemize}

\paragraph{Experimental Setup.} For our experiments, we implemented the following three algorithms: (a) Algorithm~\ref{algo:AGFW-P} with $K_q =1$, i.e., $\delta_t = \epsilon/2$ for all $t\geq 1$ (\texttt{GFW-ApproxI}), (b) Algorithm~\ref{algo:AGFW} with $\delta_t = \epsilon/2+\min_{k<t}G_k^a$ (\texttt{GFW-ApproxII}), and (c) \cite[Algorithm~1]{zhao2022analysis} (\texttt{GFW-Exact}).
Note that, for~\eqref{prob:SC-SumLogDefComp}, $\theta = d$, and $R_g = 0$ since $g(\cdot)$ is an indicator function for $\mathcal{S}$ and all three algorithms generate iterates feasible for the set $\mathcal{S}$.
The values of other parameters were set as: $\epsilon = 0.05$, $l=1$.
Note that, Algorithm~\ref{algo:AGFW-P}, and thus, \texttt{GFW-ApproxI} and \texttt{GFW-ApproxII}, stop after the bounds, $G_t^a \leq \epsilon$ and $\delta_t\leq 3\epsilon/2$, are satisfied at least $l$ times, whereas \texttt{GFW-Exact} stops if $G_t^a \leq \epsilon$~\cite{zhao2022analysis}. Let $K^u$ be defined as
\begin{equation}\label{eqn:DefKu-NE}
K^u = \begin{cases}
\lceil 10.6\Delta_0\rceil+2\left\lceil \frac{48d^2}{\epsilon}\right\rceil+2+\frac{\log(1/\overline{p})}{1-p}+2\log(1/\overline{p})
\quad \textup{for \texttt{GFW-ApproxI}}\\
\lceil 10.6\Delta_0 \rceil +2\left\lceil \frac{72d^2}{\epsilon}\right\rceil +\frac{2(2+\lceil\log_2(d/\epsilon)\rceil)}{1-p} + 2\log(1/\overline{p}) \quad \textup{for \texttt{GFW-ApproxII}}\\
\left\lceil 5.3(\Delta_0+d)\log(10.6\Delta_0)\right\rceil + \left\lceil \frac{24d^2}{\epsilon} \right\rceil\quad \textup{for \texttt{GFW-Exact}},
\end{cases}
\end{equation}
which follows from Lemma~\ref{lemma:ConvProb-CL}, Lemma~\ref{lemma:TotalConvResultSCBarrier} and~\cite[Theorem~2.1]{zhao2022analysis} respectively. Note that, in~\eqref{eqn:DefKu-NE}, $p$ is an upper bound on the probability of failure of \texttt{PLMO}, and determines its computational complexity (see Lemma~\ref{lemma:TotalConvResultSCBarrier}), while $\overline{p}$ denotes an upper bound on the probability that the algorithm stops after at most $K^u$ iterations. Thus, if $K^u$ is as defined in~\eqref{eqn:DefKu-NE}, we have that with probability at least $1-\overline{p}$, \texttt{GFW-ApproxI} and \texttt{GFW-ApproxII} stops after at most $K^u$ iterations, and with probability 1, \texttt{GFW-Exact} stops after at most $K^u$ iterations. For our experiments, we set $p = \overline{p} = 0.1$.

The computations were performed using MATLAB R2023a on a machine with 8GB RAM. For \texttt{GFW-Exact}, we used \texttt{eig} in MATLAB to implement \texttt{LMO}.
Whereas to implement \texttt{PLMO} in Algorithm~\ref{algo:AGFW-P}, we used \texttt{eigs} in MATLAB instead of Lanczos method as stated in Section~\ref{sec:SC-Barrier}. 

\paragraph{Using \texttt{eigs} in MATLAB.}
For~\eqref{prob:SC-SumLogDefComp}, the update direction $H^t$ generated at each iteration of Algorithm~\ref{algo:AGFW-P} is a rank-1 matrix $H^t = hh^T$ such that $h$ is an approximate minimum eigenvector of $\mathcal{B}^*(\nabla f(\mathcal{B}(X^t)))$.
The command \texttt{eigs} in MATLAB generates an approximate minimum eigenvalue-eigenvector pair $(\lambda,h)$ for the matrix $C= \mathcal{B}^*(\nabla f(\mathcal{B}(X^t)))$ such that $h$ is a unit vector and $\|Ch - \lambda h\| \leq \textup{tol}\times \|C\|$, where $\textup{tol}$ is a user-defined value. In case of \texttt{GFW-ApproxI}, we set $\textup{tol} = 2.22\times 10^{-16}$, which is the default value, at each iteration of the algorithm. While at each iteration $t$ of \texttt{GFW-ApproxII}, we set $\textup{tol} = \frac{\delta_t}{\|C\|}$. 
Note that, for~\eqref{prob:SC-SumLogDefComp}, the optimal objective function value of the subproblem~\eqref{prob:LinearMinimizationProblem} at iteration $t$ is $\|\mathcal{B}^*(\nabla f(\mathcal{B}(X^t)))\|$. We verify that by setting the value of $\textup{tol}$ as given above for both algorithms, \texttt{eigs} returns a unit vector $h$ such that $h^T\mathcal{B}^*(\nabla f(\mathcal{B}(X^t)))h = \|\mathcal{B}^*(\nabla f(\mathcal{B}(X^t)))\| + \eta_t\leq \|\mathcal{B}^*(\nabla f(\mathcal{B}(X^t)))\| + \delta_t$ at each iteration $t$. Thus, \texttt{eigs} returns an update direction $H^t = hh^T$ such that $\eta_t \leq \delta_t$ for all $t\geq 0$.

\paragraph{Results for \textsc{Diag} problem instances.}
The random instances generated were of two sizes (\textit{$n=500$, $d=50,100$}).
For \textsc{Diag} problem instances, we randomly initialized all three algorithms to $X^0 = \sum_{i=1}^n u_ju_j^T$ such that $u_j \sim \mathcal{N}(0,I)$ for each $j\in [n]$. For each problem instance, we performed 10 runs of each algorithm with different random initialization.
Table~\ref{table:SCGSConstEtaDiag} shows the computational results for \texttt{GFW-ApproxI}, \texttt{GFW-ApproxII} and \texttt{GFW-Exact}. The quantity $K$ denotes the number of iterations until the algorithms stop and $K^u$, defined in~\eqref{eqn:DefKu-NE}, denotes an upper bound on $K$. Note that the average is taken over 10 independent runs with different random initializations for each problem instance.

 \begin{table*}[th]
    \centering
    \caption{Computational results of using \texttt{GFW-ApproxI}, \texttt{GFW-ApproxII} and \texttt{GFW-Exact} to generate an $\epsilon$-optimal solution to \textsc{Diag} instances of~\eqref{prob:SC-SumLogDef}. Here, average is taken over 10 independent runs with different random initialization.}
    \label{table:SCGSConstEtaDiag}
    \footnotesize
    \begin{tabular}{|c| c| c| c| c| c| c|}
    \hline
\parbox[c]{0.4cm}{\raggedright $n$} &
\parbox[c]{0.4cm}{\raggedright $d$} &
\parbox[c]{1.1cm}{\raggedright Algorithm} & \parbox[c]{1.5cm}{\raggedright Avg time (secs)} & \parbox[c]{1.0cm}{\raggedright Avg $K^u$ ($\times 10^6$)} &
\parbox[c]{0.9cm}{\raggedright Avg $K$ ($\times 10^4$)} & \parbox[c]{0.8cm}{\raggedright $\sigma[K]$} \\
\hline
500&50&\texttt{GFW-ApproxI}&652.33&4.8&2.948&356.79\\
500&50&\texttt{GFW-ApproxII}&682.33&7.2&3.03&337.07\\
500&50&\texttt{GFW-Exact}&1038.1&1.21&5.054&0.95\\
\hline
500&100&\texttt{GFW-ApproxI}&2374.9&19.2&5.717&398.36\\
500&100&\texttt{GFW-ApproxII}&2548&28.8&5.965&748.17\\
500&100&\texttt{GFW-Exact}&8407.4&4.81&20.12&2.02\\
\hline
\end{tabular}
\end{table*}

From average $K^u$ and average $K$ in Table~\ref{table:SCGSConstEtaDiag}, we see that all three algorithms perform approximately two orders of magnitude better than the theoretical convergence bounds. Also, the standard deviation of $K$ shows small variation over random initialization of the algorithms.
We also see that for both random instances of varying size, \texttt{GFW-ApproxI} and \texttt{GFW-ApproxII} perform lesser number of iterations than \texttt{GFW-Exact} before stopping.
Figure~\ref{fig:OptGap1} shows the change in the optimality gap with number of iterations for the three algorithms when they were used to generate an $\epsilon$-optimal solution to a \textsc{Diag} instance of~\eqref{prob:SC-SumLogDefComp} with $n=500$, $d=50$. 
From the figure, we see that all three algorithms show similar trend in the change of optimality gap at each iteration, and after a few initial iterations, the algorithms converge linearly wrt $t$ to an $\epsilon$-optimal solution. We also observe that \texttt{GFW-ApproxI} and \texttt{GFW-ApproxII} require slightly fewer iterations than \texttt{GFW-Exact} to converge to an $\epsilon$-optimal solution.

\begin{figure}[ht]
\centering
\includegraphics[width=.6\textwidth]{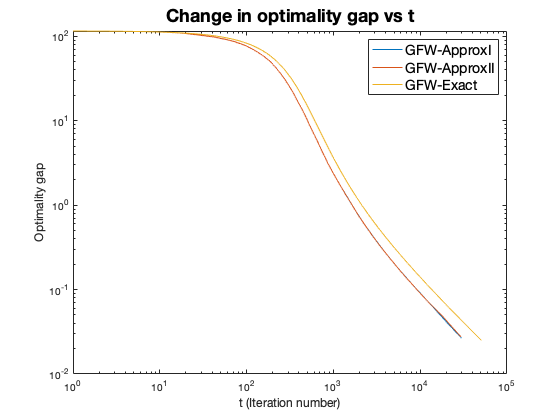}
\caption{Optimality gap vs number of iterations for \textsc{Diag} instance ($n=500$, $d=50$) of~\eqref{prob:SC-SumLogDefComp}.}\label{fig:OptGap1}
\end{figure}

\paragraph{Results for \textsc{Rnd} problem instances.}
The three algorithms were randomly initialized to $X^0$ such that $X^0 = \sum_{j=1}^n u_ju_j^T$ with $u_j\sim \mathcal{N}(0,I)$ for $j\in [n]$. For each problem instance, we performed 10 runs of each algorithm with different random initialization.
The results for \textsc{Rnd} problem instances are given in Table~\ref{table:SCGSConstEtaRnd}. In the table, the quantity $K$ denotes the number of iterations until the algorithms stop and $K^u$, defined in~\eqref{eqn:DefKu-NE}, denotes an upper bound on $K$. The average values in the table arise from the different random initial points during each run of the algorithm. 

 \begin{table*}[th]
    \centering
    \caption{Computational results of using \texttt{GFW-ApproxI}, \texttt{GFW-ApproxII} and \texttt{GFW-Exact} to generate an $\epsilon$-optimal solution to \textsc{Rnd} instances of~\eqref{prob:SC-SumLogDef}.}
    \label{table:SCGSConstEtaRnd}
    \footnotesize
    \begin{tabular}{|c| c| c| c| c| c| c|}
    \hline
\parbox[c]{0.4cm}{\raggedright $n$} &
\parbox[c]{0.4cm}{\raggedright $d$} &
\parbox[c]{1.1cm}{\raggedright Algorithm} & \parbox[c]{1.5cm}{\raggedright Avg time (secs)} & \parbox[c]{1.0cm}{\raggedright $K^u$ ($\times 10^8$)} &
\parbox[c]{0.9cm}{\raggedright Avg $K$} & \parbox[c]{0.8cm}{\raggedright $\sigma[K]$} \\
\hline
200&250&\texttt{GFW-ApproxI}&1.84&1.20&88.7&1.34\\
200&250&\texttt{GFW-ApproxII}&2.46&1.80&89.7&1.34\\
200&250&\texttt{GFW-Exact}&1.94&0.30&88.7&1.34\\
\hline
300&350&\texttt{GFW-ApproxI}&6.93&2.35&111.7&1.83\\
300&350&\texttt{GFW-ApproxII}&6.89&3.52&112.7&1.83\\
300&350&\texttt{GFW-Exact}&6.23&0.58&111.7&1.83\\
\hline
400&450&\texttt{GFW-ApproxI}&14.54&3.88&131.1&1.37\\
400&450&\texttt{GFW-ApproxII}&16.96&5.82&132.1&1.37\\
400&450&\texttt{GFW-Exact}&16.38&0.97&131.1&1.37\\
\hline
400&600&\texttt{GFW-ApproxI}&21.66&6.91&149.7&1.49\\
400&600&\texttt{GFW-ApproxII}&24.92&10.36&150.7&1.49\\
400&600&\texttt{GFW-Exact}&23.79&1.73&149.7&1.49\\
\hline
500&750&\texttt{GFW-ApproxI}&63.7&10.8&172.4&1.5\\
500&750&\texttt{GFW-ApproxII}&52.72&16.2&173.4&1.62\\
500&750&\texttt{GFW-Exact}&71.71&2.70&172.4&1.5\\
\hline
600&900&\texttt{GFW-ApproxI}&152&15.55&193.6&1.17\\
600&900&\texttt{GFW-ApproxII}&142.74&23.32&194.7&1.25\\
600&900&\texttt{GFW-Exact}&122.43&3.89&193.6&1.17\\
\hline
700&750&\texttt{GFW-ApproxI}&170.08&10.8&181.8&1.32\\
700&750&\texttt{GFW-ApproxII}&167.85&16.2&182.8&1.32\\
700&750&\texttt{GFW-Exact}&176.51&2.70&181.8&1.32\\
\hline
\end{tabular}
\end{table*}

From the values of average $K^u$ and average $K$ in Table~\ref{table:SCGSConstEtaRnd}, we see that for each problem instance, all three algorithms satisfiy the stopping criterion after $<500$ iterations.
This value is several orders of magnitude smaller than the theoretical upper bounds given in~\eqref{eqn:DefKu-NE}. Furthermore, the standard deviation of $K$ shows a very small variation over the random initialization of the algorithms. Finally, we also observed that for varying values of $n$ and $d$, all algorithms take the same number of iterations to converge to an $\epsilon$-optimal solution. Thus, for the specific generated random instances, the approximation error in solving the linear subproblem does not affect the convergence of the algorithm. Furthermore, Figure~\ref{fig:DualityGap1} shows the change $G_t^a+\delta_t$ with the number of iterations for the three algorithms (with \textsc{Rnd} instance of size $n=700$, $d=750$ as input) until the stopping criterion is satisfied. The y-axis shows the average of $G_t^a+\delta_t$ taken over 10 runs of the algorithms each with a different random initialization.
Since $\Delta_t \leq G_t^a+\delta_t$, the y-axis denotes the upper bound on the optimality gap. Note that, for \texttt{GFW-ApproxII}, we have $\delta_t = \epsilon/2+\min_{\tau<t}G_{\tau}^a$, and so, for \texttt{GFW-ApproxII}, $G_t^a+\delta_t$ has a value larger than that for \texttt{GFW-ApproxI} and \texttt{GFW-Exact}. Furthermore, all three algorithms show similar trend in the decrease in $G_t^a+\delta_t$, and consequently, $G_t^a$. We also observe that the quantity $G_t^a+\delta_t$ decreases exponentially wrt $t$. Note that, this is a much faster convergence as compared to \textsc{Diag} instances (see Figure~\ref{fig:OptGap1}).

From these experiments (for \textsc{Diag} as well as \textsc{Rnd} instances), we see that our proposed method (Algorithm~\ref{algo:AGFW-P}) has similar performance as Algorithm~\ref{algo:GFW}, i.e., the algorithm with exact linear minimization oracle.

\begin{figure}[ht]
\centering
\includegraphics[width=.6\textwidth]{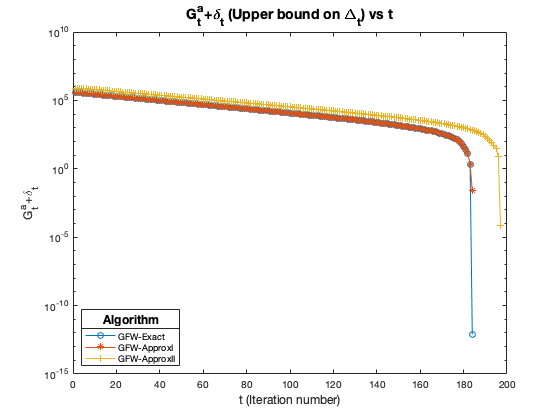}
\caption{Average $G_t^a+ \delta_t$ vs number of iterations ($t$) for \textsc{Rnd} instance ($n=700$, $d=750$) of~\eqref{prob:SC-SumLogDefComp}. Note, since $\Delta_t \leq G_t^a+\delta_t$ with probability at least $1-p^l$, y-axis denotes the upper bound on optimality gap $\Delta_t$.}\label{fig:DualityGap1}
\end{figure}

\section{Discussion}
In this paper, we proposed a first-order algorithm, based on the Frank-Wolfe method, with provable approximation guarantees for composite optimization problems with $\theta$-logarithmically-homogeneous, $M$-self-concordant barrier function in the objective and a convex, compact feasible set. We showed that it is enough to solve the linear subproblem at each iteration approximately making Algorithm~\ref{algo:AGFW} more practical while also reducing the complexity of each iteration of the algorithm. We also showed that Algorithm~\ref{algo:AGFW} can generate an $\epsilon$-optimal solution to~\eqref{eqn:SC-gendef} when the oracle generates an approximate solution with either additive or relative approximation error. 
We analyzed Algorithm~\ref{algo:AGFW} for two different settings of the value of $\delta_t$. However, it is possible that alternate techniques to set the value of $\delta_t$ could result in potentially faster convergence. Since the convergence of the algorithm depends on the initial optimality gap, we can potentially use warm starting techniques to speed up the method.
We further proposed Algorithm~\ref{algo:AGFW-P} which generates an $\epsilon$-optimal solution to~\eqref{eqn:SC-gendef} with high probability when the linear minimization oracle has a nonzero probability of failure.
Finally, we performed numerical experiments to show that the convergence analysis of our method holds true. The results of numerical experiments confirmed that the convergence bounds are satisfied for random instances of~\eqref{prob:SC-SumLogDef}.

\paragraph{Acknowledgements.} This material is based upon work supported by the National Science Foundation under Grant No. DMS-1929284 while the author, Nimita Shinde, was in residence at the Institute for Computational and Experimental Research in Mathematics in Providence, RI, during the Discrete Optimization: Mathematics,  Algorithms, and Computation semester program.

\appendix

\section{Preliminary results}
\begin{lemma}\label{lemma:RelationDtGt}
Let the quantities $R_g$, $G_t^a$ and $D_t$ be defined in~\eqref{eqn:DefRh},~\eqref{eqn:DefGkApprox} and~\eqref{eqn:DefDt} respectively. 
For any $t\geq 0$, if $G_t^a \geq 0$, it holds that
\begin{equation}\label{eqn:RelationDkGk}
D_t \leq G_t^a + \theta + R_g.
\end{equation}
\end{lemma}

\begin{proof}
\citet{zhao2022analysis} state the relationship between $D_t$ and the duality gap $G_t$, which is given as $D_t \leq G_t +\theta +R_g$~\cite[Proposition~2.3]{zhao2022analysis}. 
The proof of~\eqref{eqn:RelationDkGk} is adapted from the proof of~\cite[Proposition~2.3]{zhao2022analysis}.

From the definition of $D_t$, we have
\begin{equation}\label{eqn:DkGkProof1}
\begin{split}
D_t^2 &= \|\mathcal{B}(h^t-x^t)\|_{\mathcal{B}(x^t)}^2\\ &= \langle \mathcal{H}(\mathcal{B}(x^t))\mathcal{B}(h^t), \mathcal{B}(h^t)\rangle - 2\langle \mathcal{H}(\mathcal{B}(x^t))\mathcal{B}(x^t), \mathcal{B}(h^t)\rangle + \langle \mathcal{H}(\mathcal{B}(x^t))\mathcal{B}(x^t), \mathcal{B}(x^t)\rangle.
\end{split}
\end{equation}
We now bound $U^t = \langle \mathcal{H}(\mathcal{B}(x^t))\mathcal{B}(h^t), \mathcal{B}(h^t)\rangle$ as
\begin{equation}\label{eqn:DkGkProof2}
\begin{split}
U^t &\underset{(i)}{\leq} \langle \nabla f(\mathcal{B}(x^t)),\mathcal{B}(h^t)\rangle^2\\
&= [\langle \nabla f(\mathcal{B}(x^t)),\mathcal{B}(h^t-x^t)\rangle + \langle \nabla f(\mathcal{B}(x^t)),\mathcal{B}(x^t)\rangle]^2\\
&\underset{(ii)} {=}[-G_t^a - g(h^t) +g(X^t) +\langle \nabla f(\mathcal{B}(x^t)),\mathcal{B}(x^t)\rangle]^2\\
&\underset{(iii)}{=} [-G_t^a - g(h^t) + g(x^t) - \theta]^2,
\end{split}
\end{equation}
where (i) follows from~\eqref{eqn:SCProp2}, (ii) follows from the definition of $G_t^a$ (given in~\eqref{eqn:DefGkApprox}), and (iii) follows from~\eqref{eqn:SCProp4}.
For simplicity, define $V^t = - 2\langle \mathcal{H}(\mathcal{B}(x^t))\mathcal{B}(x^t), \mathcal{B}(h^t)\rangle + \langle \mathcal{H}(\mathcal{B}(x^t))\mathcal{B}(x^t), \mathcal{B}(x^t)\rangle$.
Then, from~\eqref{eqn:SCProp3},~\eqref{eqn:SCProp4} and~\eqref{eqn:DefGkApprox}, it follows that
\begin{equation}\label{eqn:DkGkProof3}
\begin{split}
V^t &\underset{(i)}{=} 2\langle \nabla f(\mathcal{B}(x^t)),\mathcal{B}(h^t-x^t)\rangle + \langle \nabla f(\mathcal{B}(x^t)),\mathcal{B}(x^t)\rangle\\
&\underset{(ii)}{=} -2G_t^a -2g(h^t)+2g(x^t) + \langle \nabla f(\mathcal{B}(x^t)),\mathcal{B}(x^t)\rangle\\
&\underset{(iii)}{=} -2G_t^a -2g(h^t)+2g(x^t) - \theta,
\end{split}
\end{equation}
where (i) follows from~\eqref{eqn:SCProp3}, (ii) follows from~\eqref{eqn:DefGkApprox}, and (iii) follows from~\eqref{eqn:SCProp4}.
Combining~\eqref{eqn:DkGkProof1},~\eqref{eqn:DkGkProof2} and~\eqref{eqn:DkGkProof3},
\begin{equation*}
\begin{split}
D_t^2 &\leq [(G_t^a+\theta)+(g(h^t)-g(x^t))]^2 -2G_t^a -2g(h^t)+2g(x^t) - \theta\\
&= (G_t^a+\theta)^2 + 2(G_t^a+\theta)(g(h^t)-g(x^t))+ (g(h^t)-g(x^t))^2-2G_t^a -2g(h^t)+2g(x^t) - \theta\\
&\underset{(i)}{\leq} (G_t^a+\theta)^2+ 2(G_t^a+\theta)(g(h^t)-g(x^t))+ (g(h^t)-g(x^t))^2- 2(g(h^t)-g(x^t))\\
&= (G_t^a+\theta)^2 + 2(G_t^a+\theta-1)(g(h^t)-g(x^t))+(g(h^t)-g(x^t))^2\\
&\underset{(ii)}{\leq} (G_t^a+\theta)^2+ 2R_g(G_t^a+\theta - 1)+R_g^2\\
&\leq (G_t^a+\theta)^2  + 2R_g(G_t^a+\theta)+R_g^2\\
&= (G_t^a+\theta + R_g)^2,
\end{split}
\end{equation*}
where (i) uses that the fact that $G_t^a \geq 0$ and $\theta \geq 0$, and (ii) uses the definition of $R_g$~\eqref{eqn:DefRh}.
\end{proof}

\bibliography{refLong,references}

\end{document}